\documentclass[production,11pt]{article}
\usepackage{amsmath}
\usepackage{amssymb}
\usepackage{amsthm}
\usepackage{amsfonts}
\usepackage{latexsym}

\theoremstyle{plain}
\newtheorem{thm}{Theorem}
\newtheorem{cor}{Corollary}[section]
\newtheorem{lem}{Lemma}[section]

\theoremstyle{definition}
\newtheorem{defn}{Definition}[section]

\theoremstyle{remark}

\newtheorem{rem}{Remark}[section]

\title{Asymptotics of Szeg\"{o} kernels\\ under Hamiltonian torus actions}
\author{Roberto Paoletti\footnote{\noindent{\bf Address:}
Dipartimento di Matematica e Applicazioni, Universit\`a degli Studi
di Milano Bicocca, Via R. Cozzi 53, 20125 Milano,
Italy; {\bf e-mail}: roberto.paoletti@unimib.it }}
\date{}

\begin{document}

\maketitle

\begin{abstract}
Let $X$ be the circle bundle associated to a positive line bundle on a
complex projective (or, more generally, compact symplectic)
manifold. The Tian-Zelditch expansion on $X$ may be seen as a local manifestation of the
decomposition of the (generalized) Hardy space $H(X)$ into isotypes for the $S^1$-action. More generally, given a compatible
action of a compact Lie group, and under general assumptions guaranteeing finite dimensionality of isotypes,
we may look for asymptotic expansions locally reflecting the equivariant decomposition of $H(X)$ over the irreducible
representations of the group. We focus here on the case of compact tori.
\end{abstract}



\section{Introduction}

Let $M$ be a connected d-dimensional complex projective manifold, with an ample line bundle $A$
on it. Thus there exists an Hermitian metric $h$ on $A$, such that the curvature form
of the unique
connection $\nabla _A$ compatible with both the holomorphic structure and the metric is $\Theta=-2i\,\omega$, where
$\omega$ is a K\"{a}hler form;
we shall adopt the volume form $dV_M=(1/\mathrm{d}!)\,\omega^{\wedge \mathrm{d}}$
on $M$.
Each space of global holomorphic sections $H^0\left(M,A^{\otimes k}\right)$
is naturally an Hermitian vector space, and we may form the Hilbert
space direct sum $H(A)=:\bigoplus_{k\ge 0}H^0\left(M,A^{\otimes k}\right)$.

Now suppose that $\mu$ is an holomorphic
Hamiltonian action of the g-dimensional compact torus $\mathbb{T}=\mathbb{T}^\mathrm{g}$ on $(M,2\omega)$.
Assume, in addition, that $\mu$ can be linearized to an action $\widetilde{\mu}$ on $A$; after averaging,
we may suppose that $h$ and $\nabla_A$ are $\mathbb{T}$-invariant.  This lifting induces natural unitary representations
of $\mathbb{T}$ on $H^0\left(M,A^{\otimes k}\right)$, $k=0,1,2,\ldots$, and therefore on $H(A)$. Thus we can unitarily
and equivariantly decompose $H(A)$ over the irreducible representations of $\mathbb{T}$, which are of course just its characters:
\begin{equation}
\label{eqn:decomp-A}
H(A)=:\bigoplus _{\varpi \in \mathbb{Z}^{\mathrm{g}}}H_\varpi(A).
\end{equation}
For example, when $\mathbb{T}^1=S^1$ acts trivially on $M$, with constant moment map equal to $1$, we have
$H(A)_k=H^0\left(M,A^{\otimes k}\right)$. In general, however, $H(A)_\varpi$ needn't be contained in a space of
global sections, nor need it be finite-dimensional. Nonetheless, finite dimensionality
is ensured under the additional hypothesis that $\mathbf{0}\not\in \Phi(M)$, where $\Phi:M\rightarrow \mathfrak{t}^\vee$
is the moment map (here $\mathfrak{t}=\mathrm{Lie}(\mathbb{T})$).
The latter requirement may be seen as a sort of \lq homogeneous properness\rq,
since it implies that for any $\varpi\in \mathfrak{t}^\vee\setminus\{\mathbf{0}\}$ the inverse image in $M$
of the ray $\mathbb{R}_+\cdot \varpi$ is compact.

Under this assumption,
the orthogonal projector $P_\varpi:H(A)\rightarrow H(A)_\varpi$ is a smoothing operator; restricting to the diagonal,
we obtain a well-defined smooth function on $M$, which pictorially describes the \lq local contribution to $H(A)_\varpi$\rq.
In particular, a point-wise estimate of the latter leads to a global estimate on $\dim \big(H(A)_\varpi\big)$. In the classical
case of the standard circle action on $A$, thus with $\Phi=1$,
this point-wise estimate is the celebrated Tian-Zelditch expansion.

As we have remarked, in the latter basic case the underlying action on the base manifold is trivial, while the
lifted action on the line bundle is not; this accounts for the non-trivial equivariant decomposition
described by the TYZ expansion. A basic theme in geometric quantization of Hamiltonian group actions, that
we recall explicitly below, is that the lifted action is infinitesimally described by a combination of
the horizontal lift of the action on the base manifold, and of the structure circle action, with a weight
controlled by the moment map. In particular, different choices of $\Phi$ for the same $\mu$
determine different lifts $\widetilde{\mu}$, whence different unitary representations on $H(A)$.

The present
paper is devoted to the case of a general torus action.
One motivation is to provide some local counterpart
to the general philosophy of \cite{gs2}, where one considers the various possible reductions of a symplectic
cone (in particular, if $\Sigma$ is the symplectic cone sprayed by the connection, the base manifold $M$ is the
reduction associated to the standard circle action).

Let us clarify the issue by giving some explicit examples.

Consider first the unitary representation
$T^1\times \mathbb{C}^2\rightarrow \mathbb{C}^2$ given by $t\cdot (z_0,z_1)=:\left(t\,z_0,t^s\,z_1\right)$
for some integer $s\ge 1$.
This representation determines an action on $\mathbb{P}^1$ as well as a built-in linearization
to the hyperplane line bundle $A=\mathcal{O}_{\mathbb{P}^1}(1)$, associated to the moment map
$$
\Phi:\mathbb{P}^1\rightarrow \mathbb{R},\,\,
[z_0:z_1]\mapsto \frac{|z_0|^2+s\,|z_1|^2}{|z_0|^2+|z_1|^2};
$$
here the underlying K\"{a}hler structure on $\mathbb{P}^1$ is of course the Fubini-Study form.
In particular, $\Phi(\mathbb{P}^1)=[1,s]$.
The action on the dual $A^\vee$ is simply the induced action on the incidence correspondence
$\{(\mathbf{v},[\mathbf{z}]): \mathbf{v}\wedge \mathbf{z}=\mathbf{0}\}\subseteq \mathbb{C}^2\times \mathbb{P}^1$,
and it obviously preserves the unit sphere $X=S^3\subseteq \mathbb{C}^2$.

Clearly, $H(A)_k$ is the span of those monomials $Z_0^a\,Z_1^b$ with $a+b\,s=k$; if $k=s\,k_0+k_1$,
with $0\le k_1<s$, then $(a,b)=(k_1,k_0),\,(k_1+s,k_0-1),\ldots,(k,0)$; in particular, $\dim H_k(A)=k_0+1$.
In addition, if $a=k_1+js$ and $b=k_0-j$, then
$$
a+b=k_1+k_0+j\,(s-1);
$$
therefore, if $s\ge 2$ then no two of these generators are sections of $A^{\otimes l}$ for the same $l$,
so that $H(A)_k$ may not be interpreted as a space of sections of some power of $A$. If $s=1$, of course, we
fall back on the standard case of the structure circle action.

Next, we consider the unitary representation $\mathbb{T}^1\times \mathbb{C}^3\rightarrow \mathbb{C}^3$ given by, say,
$t\cdot (z_0,z_1,z_3)=:\left(t\,z_0,t^2\,z_1,\,t^3\,z_2\right)$. Again, this descends to an action on $\mathbb{P}^2$
with a built-in linearization to the hyperplane line bundle $A$, and the associated moment map $\Phi:\mathbb{P}^2\rightarrow
\mathbb{R}$ is
$$
\Phi\big([z_0:z_1:z_2]\big)=:\frac{|z_0|^2+2\,|z_1|^2+3\,|z_2|^2}{|z_0|^2+|z_1|^2+|z_2|^2}.
$$
In particular, $\Phi\left(\mathbb{P}^2\right)=[1,3]$.
Now $H_k(A)$ is the span of those monomials $Z_0^a\,Z_1^b\,Z_2^c$ such that
$a+2\,b+\,3c=k$. Thus we need to have $b\le \lfloor k/2\rfloor$, $c\le \lfloor (k-2b)/3\rfloor$ and then
$a=k-2b-3c$. It follows that $\dim H_k(A)=O\left(k^2\right)$ as $k\rightarrow +\infty$. Again, $a+b+c$ is not constant
over the set of these monomials, so that $H(A)_k$ is not a space of sections of some power of $A$.

When a linearized Hamiltonian action leaves invariant a projective submanifold of the base, there is a linearized
action induced by restriction; this is Hamiltonian with respect to the restriction of the K\"{a}hler
structure on the ambient space.

For example, the action on $\mathbb{P}^2$
induced by the previous representation
leaves invariant the smooth conic $S\,: \,X_1^2-X_0\,X_2=0$,
and therefore there is an induced linearization to the restricted hyperplane line bundle on $S$; the moment map
on $S$ is simply the restricted moment map and therefore it is positive. Now as a projective manifold
$S$ is isomorphic to $\mathbb{P}^1$ under the
Veronese embedding $\nu_{1,2}:\mathbb{P}^1\rightarrow \mathbb{P}^2$, $[u:s]\mapsto \left[u^2:u\,s:s^2\right]$, and
the induced action on $\mathbb{P}^1$ is given by $t\cdot [u:s]=\left[t^{1/2}u:t^{3/2}s\right]$;
the action is well-defined, although $t^{1/2}$ is not. Under this isomorphism, the pull-back
$A_S=\nu_{1,2}^*(A)\cong \mathcal{O}_{\mathbb{P}^1}(2)$,
and the moment map for this linearized action is
$$
\Phi_S=:\Phi\circ \nu_{1,2}:\mathbb{P}^1\rightarrow \mathbb{R},\,\,\,\,\,\,[u:s]\mapsto
\frac{|u|^4+2\,|u|^2\,|s|^2+3\,|s|^4}{|u|^4+|u|^2\,|s|^2+|s|^4}.
$$
Clearly, $\Phi_S\left(\mathbb{P}^1\right)=[1,3]$.

Under the Veronese embedding,  $H(A_S)_k$ may be identified with the span of all monomials
$Z_0^a\,Z_1^b$ subject to the conditions $a+b=2r$ and $r+b=k$ for some integer $r$.
It follows that $0\le b\le \lfloor 2k/3\rfloor$, $a=2k-3b$; hence $\dim H(A_S)_k= \lfloor 2k/3\rfloor$.
In particular, $a+b=2\,(k-b)$ is different for all these generators, so that $H(A_S)_k$
picks up 1-dimensional contributions from $H^0\left(\mathbb{P}^1,\mathcal{O}_{\mathbb{P}^1}(2\,l)\right)$ for
$(k/3)-1\le l\le k$.

Obviously, these considerations may be generalized to any unitary representation of $\mathbb{T}^1$ on $\mathbb{C}^{\mathrm{d}+1}$
of the form $t\cdot (z_0,\ldots,z_\mathrm{d})=:\left(t^{\ell_0}\,z_0,\ldots,t^{\ell_\mathrm{d}}\,z_\mathrm{d}\right)$, where
$\ell _j>0$ for every $j$;
the moment map has image the interval $[l,L]$, where
$l=\min (\ell_j)$, $L=\max (\ell_j)$. They all determine linearized actions on the polarized pair
$\left(\mathbb{P}^\mathrm{d},\mathcal{O}_{\mathbb{P}^\mathrm{d}}(1)\right)$, as well as on any invariant submanifold
polarized by the restriction of the hyperplane bundle.

For instance, $t\cdot (z_0,z_1,z_2,z_3)=:\left(t\,z_0,t^2\,z_1,t^3\,z_2,t^4\,z_3\right)$ induces an
action on $\mathbb{P}^3$
that leaves invariant the quadric defined by $X_1\,X_2-X_0\,X_3=0$, which is isomorphic to $\mathbb{P}^1\times
\mathbb{P}^1$ under the Segre embedding. The corresponding linearization then also restricts.
On the other hand,  the action on $\mathbb{P}^4$
induced by $t\cdot (z_0,z_1,z_2,z_4)=:\left(t\,z_0,t^2\,z_1,t^3\,z_2,t^4\,z_3,t^5\,z_4\right)$
leaves invariant the smooth quadric hypersurface defined by $X_0\,X_4+X_1\,X_3+X_2^2=0$.

Noteworthy is the example associated to
$$t\cdot (z_0,z_1,z_2,z_3,z_4,z_5)=:\left(t\,z_0,t^2\,z_1,t^3\,z_2,t^3\,z_3,t^4\,z_4,t^5\,z_5\right),
$$
which may
be restricted to the invariant Klein quadric $X_0\,X_5-X_1\,X_4+X_2\,X_3=0$ in $\mathbb{P}^5$.
Under the Pl\"{u}cker embedding, the latter is
isomorphic to the Grassmanian $G(2,4)$ of two-dimensional vector subspaces in $\mathbb{C}^4$, and with this identification
the restricted action on $G(2,4)$ is the one associated to the unitary action
$t\cdot (x,y,z,u)\mapsto \left(x,t\,y,t^2\,z,t^3\,u\right)$ on $\mathbb{C}^4$. In terms of Pl\"{u}cker
coordinates, the moment map for the linearized action pulled back to $G(2,4)$ is
$$
\Phi_{G(2,4)}=\frac{|p_{01}|^2+2\,|p_{02}|^2+3\,|p_{03}|^2+3\,|p_{12}|^2+4\,|p_{13}|^2+5\,|p_{23}|^2}
{|p_{01}|^2+|p_{02}|^2+|p_{03}|^2+|p_{12}|^2+|p_{13}|^2+|p_{23}|^2}.
$$

These examples may be generalized to higher dimensions of the group as well as of
the manifold $M$. Let us consider a couple of cases with $\mathrm{g}=2$.

Consider first the linearized action of $\mathbb{T}^2=S^1\times S^1$ on
$\left(\mathbb{P}^1,\mathcal{O}_{\mathbb{P}^1}(1)\right)$ induced by
the unitary representation
$(t,s)\cdot (z_0,z_1)=(t\,z_0,s\,z_1)$. The moment map $\Phi:\mathbb{P}^1\rightarrow \mathbb{R}^2$
is
$$\Phi\big([z_0:z_1]\big)=:\left(\frac{|z_0|^2}{|z_0|^2+|z_1|^2},\frac{|z_1|^2}{|z_0|^2+|z_1|^2}\right),$$
and is therefore never zero; the image is the segment $\big[(1,0),(0,1)\big]$ on the line $x+y=1$.
Choose $\varpi=(\varpi_0,\varpi_1)\in \mathbb{Z}^2\setminus \{(0,0)\}$; with
$A=\mathcal{O}_{\mathbb{P}^1}(1)$, we see that $H(A)_{k\varpi}$ is the span of $Z_0^{k\varpi_1}\,Z_1^{k\varpi_2}$
if $\varpi_0,\varpi_1\ge 0$, the null space otherwise.
Thus, in this case
$H(A)_{k\varpi}\subseteq H^0\left(\mathbb{P}^1,\mathcal{O}_{\mathbb{P}^1}\big(k(\varpi_1+\varpi_2)\big)\right)$.
Introducing the appropriate normalization constant, the orthogonal projector onto $H(A)_{k\varpi}$ has kernel
$$
\widetilde{\Pi}_{k\varpi}(Z,W)=:\frac{\big(k\,|\varpi|+1)!}{\pi\,(k\varpi_1)!\,(k\varpi_2)!}\cdot
(Z_0\,\overline{W}_0)^{k\varpi_1}\cdot (Z_1\,\overline{W}_1)^{k\varpi_2},
$$
where $|\varpi|=|\varpi_0|+|\varpi_1|$.
Suppose $\varpi_0,\varpi_1> 0$, so that the ray $\mathbb{R}_+\varpi$ meets
$\Phi\left(\mathbb{P}^1\right)$ transversely. Using the Stirling formula, we get
\begin{eqnarray*}
\widetilde{\Pi}_{k\varpi}(Z,Z)&=&\frac{\big(k\,|\varpi|+1)!}{\pi\,(k\varpi_1)!\,(k\varpi_2)!}\cdot
|Z_0|^{2k\varpi_0}\cdot |Z_1|^{2k\varpi_1}\\
&\sim& c_\varpi\,\sqrt{\frac{|\varpi|}{|\varpi_0|\,|\varpi_1|}}\cdot k^{1/2}\,
\left(\frac{|\varpi|}{\varpi_0}\,|Z_0|^2\right)^{k\varpi_0}\,\left(\frac{|\varpi|}{\varpi_1}\,|Z_1|^2\right)^{k\varpi_1}.
\end{eqnarray*}
We can view $\widetilde{\Pi}_{k\varpi}(Z,Z)$
as a function of $[z_0:z_1]\in\mathbb{P}^1$ by restricting it to the unit sphere
$S^3\subseteq \mathbb{C}^2$; then it is $O\left(k^{-\infty}\right)$ unless
$$\Phi\big([z_0:z_1]\big)=\big(\varpi_0/|\varpi|,\varpi_1/|\varpi|\big),$$
that is, unless
$\Phi\big([z_0:z_1]\big)\in \mathbb{R}_+\varpi$. In the latter case, on the other hand, it is
$O\left(k^{1/2}\right)$.

The unitary representation $(t,s)\cdot (z_0,z_1,z_2)=\big(t\,z_0,s\,z_1,ts\,z_2\big)$
of $\mathbb{T}^2$ on $\mathbb{C}^3$ descends to a linearized Hamiltonian action on
$\left(\mathbb{P}^2,\mathcal{O}_{\mathbb{P}^2}(1)\right)$, with moment map
$$\Phi\big([z_0:z_1:z_2]\big)=:\left(\frac{|z_0|^2+|z_2|^2}{|z_0|^2+|z_1|^2+|z_2|^2},
\frac{|z_1|^2+|z_2|^2}{|z_0|^2+|z_1|^2+|z_2|^2}\right).
$$
Again, $\Phi$ is never vanishing, and its image is the full triangle with vertexes
$(1,0),\,(1,1),\,(0,1)$.

If $\varpi=(\varpi_0,\varpi_1)\in \mathbb{Z}^2\setminus \{(0,0)\}$, then
$H(A)_{k\varpi}$ is the span of those monomials
$Z_0^a\,Z_1^b\,Z_2^c$ for which $a+c=k\varpi_0$, $b+c=k\varpi_1$ if $\varpi_0,\varpi_1\ge 0$, the null space otherwise.
Supposing, say, $\varpi_1\le \varpi_0$ we may take $k\,\varpi_0\ge a\ge k\,(\varpi_0-\varpi_1)$,
$c=k\,\varpi_0-a$, $b=a-k\,(\varpi_0-\varpi_1)$. Thus, $\dim H(A)_k=1+k\,\varpi_1=O(k)$ if $\varpi_1>0$.
Furthermore, $a+b+c=a+k\,\varpi_1$, so that $H(A)_{k\varpi}$ picks up one-dimensional contributions
from $H^0\left(\mathbb{P}^2, \mathcal{O}_{\mathbb{P}^2}(l)\right)$ for any
$l=k\varpi_0,\ldots,k\,|\varpi|$.

Let us return to the general problem in point.

Tian-Zelditch expansions appeared in  \cite{t}, \cite{z}, \cite{c},
and have since been studied extensively by many authors; as is well-known,
there exist to date various approaches to the asymptotics for the Bergman-Szeg\"{o} kernel of a
positive line bundle on a complex projective (or, more generally, compact symplectic) manifold.
In this paper, we shall specifically
build on ideas and techniques from \cite{z}, \cite{bsz}, \cite{sz},
based on the theory of \cite{bs}.
The same approach was applied to scaling asymptotics of Szeg\"{o} kernels in the presence
of Hamiltonian symmetries in
\cite{p-jsg}, but from a different perspective (see the discussion at the end of this
introduction).

Thus, following \cite{z}, we shall lift the analysis to the unit circle bundle $X\subseteq A^\vee$, where
$A^{\vee}=A^{-1}$ is the dual (or inverse) line bundle, and work with the Hardy space $H(X)$.
This point of view
seems quite
intrinsic to our problem, since the equivariant spaces in point are generally not spaces of sections
of powers of $A$.

Now $X$ is a principal $S^1$-bundle on $M$, with projection
$\pi:X\rightarrow M$ (we shall generally denote the circle by $S^1$ when it acts on $X$ in the standard manner, and
by $\mathbb{T}^1$ when it acts by $\widetilde{\mu}$);
for example, when $(M,A)=\left(\mathbb{P}^\mathrm{d},\mathcal{O}_{\mathbb{P}^\mathrm{d}}(1)\right)$ and $\omega$
is the Fubini-Study form, $X$ may be identified with the unit sphere $S^{2\mathrm{d}+1}\subseteq \mathbb{C}^{\mathrm{d}}$,
and $\pi$ with the Hopf map.

If $\alpha$ is the normalized connection form on $X$, then $(X,\alpha)$ is a contact manifold and $dV_X=:(1/2\pi)\,\alpha\wedge
\pi^*(dV_M)$ is a volume form on $X$.
In the case of projective space, the contact structure on $S^{2\mathrm{d}+1}$
is the standard one that it inherits as the boundary of a strictly pseudoconvex domain of $\mathbb{C}^{\mathrm{d}+1}$;
similarly, in the general case $X$ is the boundary a strictly pseudoconvex domanin in $A^\vee$,
given by the unit disc bundle.
As is well-known, if $H_k(X)$ is the $k$-th isotype of the Hardy space
 $H(X)\subseteq L^2(X)$ under the bundle $S^1$-action,
 there is a natural unitary isomorphism $H_k(X)\cong H^0\left(M,A^{\otimes k}\right)$
for every $k$; thus $H(X)\cong H(A)$.
If $\Pi_k \in \mathcal{C}^\infty(X\times X)$ is the level-$k$ Szeg\"{o} kernel, that is, the distributional kernel of the
orthogonal projector $L^2(X)\rightarrow H_k(X)$, the \lq classical\rq \,TYZ expansion states that as $k\rightarrow +\infty$
$$
\Pi_k(x,x)\sim \left(\frac k\pi\right)^\mathrm{d}\cdot \left (1+\sum_{j=1}^{+\infty}a_j(m)\,k^{-j}\right),
$$
where $m=\pi(x)$ and each $a_j$ is a differential polynomial in the metric.

This representation-theoretical description may be extended to a general
$\widetilde{\mu}$, as follows.
Since $h$ is $\widetilde{\mu}$-invariant, so is $X$; hence $\mathbb{T}$ acts on it as a group of contactomorphisms, and the ensuing
unitary representation on $L^2(X)$ preserves $H(X)$.
We shall write the counterpart of (\ref{eqn:decomp-A}) as
\begin{equation}
\label{eqn:decomp-X}
H(X)=:\bigoplus _{\varpi \in \mathbb{Z}^{\mathrm{g}}}\widetilde{H}_\varpi(X),
\end{equation}
where obviously $\widetilde{H}_\varpi(X)\cong H_\varpi(A)$. For the bundle circle action
$\widetilde{H}_k(X)= H_k(X)$, but for a general $\mathbb{T}^1$-action these are different subspaces; actually, in general
$\widetilde{H}_\varpi(X)\cap H_k(X)\neq \{0\}$ for some fixed $\varpi$ and several $k$'s, and so $\widetilde{H}_\varpi(X)$
is not (isomorphic to) a space of global sections of some power of $A$. Now for any
$\varpi\in \mathbb{Z}^\mathrm{g}$ we may consider the
level-$\varpi$ Szeg\"{o} kernel $\widetilde{\Pi}_\varpi\in \mathcal{D}'(X\times X)$, that is, the
distributional kernel of the orthogonal projector $L^2(X)\rightarrow \widetilde{H}_\varpi (X)$; if $\mathbf{0}\not\in \Phi(M)$,
$\widetilde{\Pi}_\varpi$ is smooth, and its diagonal restriction is well-defined on $M$ (the given choices provide
natural identifications between densities, half-densities and functions). The modified Tian-Zelditch expansion we are aiming at is
an asymptotic expansion for $\widetilde{\Pi}_{k\varpi}(x,x)$ as $k\rightarrow +\infty$.

As we have remarked, symplectically
the requirement that $\mathbf{0}\not\in \Phi(M)$ is a form of properness of the moment map;
analytically, it is really an ellipticity condition on the action, so that these actions
might be reasonably called \lq elliptic\rq.
To see this, let $\mathfrak{t}$ be the Lie algebra of $\mathbb{T}^\mathrm{g}$,
and for any $\xi\in \mathfrak{t}$ of $\mathbb{T}^\mathrm{g}$
let $\xi_M$ and $\xi_X$ be the smooth vectors induced by $\xi$ on $M$ and $X$ under $\mu$
and $\widetilde{\mu}$, respectively.
In particular, in standard notation we shall write
$\partial/\partial\theta$ for the generator of the structure circle action on $X$.
Then the relation between $\xi_X$ and $\xi_M$ is expressed by the relation
$\xi_{X}=\xi_{M}^{\sharp}-\langle \Phi,\xi\rangle\cdot \partial/\partial\theta$,
where the suffix $\sharp$ denotes the horizontal lifting for the connection.
Now let $(\xi_1,\ldots,\xi_\mathrm{g})$
be a basis of the Lie algebra,
and let $\overline{\partial} _b$ be the CR operator on $X$.
Then imposing $\mathbf{0}\not\in \Phi(M)$ amounts to requiring that
$(\overline{\partial} _b,\xi_{1X},\ldots,\xi_{\mathrm{g}X})$ be jointly elliptic.
The same holds of $\big(\overline{\partial} _b,\xi_{1X}-\langle\varpi,\xi_1\rangle,\ldots,
\xi_{\mathrm{g}X}-\langle\varpi,\xi_{\mathrm{g}}\rangle \big)$ for any
$\varpi\in \mathfrak{t}^\vee$ (the Lie coalgebra);
finite dimensionality of isotypes follows from this and Theorem  19.5.1 of \cite{hor-III}.
Nonetheless, a direct elementary symplectic proof will be given in \S 2,
based on the theory
of \cite{gs1} (this applies to non-toric actions as well).

For ease of exposition, we shall consider circle actions separately. In this case, we shall provide
scaling asymptotics akin to those in \cite{sz}. Thus, we are dealing for the
time being with a Hamiltonian action $\mu:\mathbb{T}^1\times M\rightarrow M$ on $(M,2\omega)$,
with (say) moment map $\Phi>0$, and a linearization $\widetilde{\mu}:\mathbb{T}^1\times A\rightarrow A$;
$\widetilde{\mu}$ will also be the contact action on $X$, and we write $\widetilde{\mu}_t(x)$ for $\widetilde{\mu}(t,x)$.
Also, for any $k\in \mathbb{Z}$ we set
\begin{equation}
\label{eqn:equivariant-piece}
\widetilde{H}_k(X)=\left\{s\in H(X)\,:\,s\left(\widetilde{\mu}_{t^{-1}}(x)\right)=t^k\,s(x)\,\forall\,t\in \mathbb{T}^1,
\,x\in X\right\},
\end{equation}
and $\widetilde{\Pi}_k$ is the (smooth) Schwartz kernel of the orthogonal projector $L^2(X)\rightarrow \widetilde{H}_k(X)$.

Since $\Phi>0$, $\widetilde{\mu}$ is locally free on $X$, hence
the stabilizer subgroup $T_x\subseteq \mathbb{T}^1$
of any $x\in X$ is finite; as it depends only on $m=\pi(x)$, we shall write $T_m$ for $T_x$.
In fact, if $\xi_M$ and $\xi_X$ are the vector fields on $M$ and $X$ induced by $\mu$ and $\widetilde{\mu}$,
respectively, then in Heisenberg local coordinates (\S \ref{subsect:heis})
$\xi_X(x)=\big(-\Phi(m),\xi_M(m)\big)$ if $m=\pi(x)$.
The tangent space to the $\widetilde{\mu}$-orbit through $x$ is
$T_x\left(\mathbb{T}^1\cdot x\right)=\mathrm{span}\big\{\xi_X(x)\big\}$.
Let $\xi_X(x)^\perp\subseteq T_xX$ be the orthocomplement to $\xi_X(x)$.

For the next definition, recall that the connection yields at any $x\in X$ a built-in unitary isomorphism
$T_xX\cong \mathbb{R}\times T_mM$, so that any $\upsilon\in T_xX$ can be intrinsically
decomposed as $\upsilon=(\theta,\mathbf{v})$.

\begin{rem}
Suppose $x\in X$, $m=\pi(x)$, and $t\in T_m=T_x$ (the stabilizer of $x$ in $\mathbb{T}^1$ under $\widetilde{\mu}$).
If $\upsilon=(\theta,\mathbf{v})\in T_xX$,
then $d_x\widetilde{\mu}_t(\upsilon)=\big(\theta,d_m\mu_t(\mathbf{v})\big)$ (Lemma \ref{lem:heis-tg}).
\end{rem}

\begin{defn}
Let $E:TX\oplus TX\rightarrow \mathbb{C}$ be defined as follows: if $x\in X$ and
$\upsilon_j=(\theta_j,\mathbf{v}_j)\in T_xX$,
$j=1,2$, then
\begin{eqnarray*}
E(\upsilon_1,\upsilon_2)
&=:&\frac{1}{\Phi(m)}\left\{i\,\left[\frac{\theta_2-\theta_1}{\Phi(m)}\,\omega_m\big(\xi_M(m),\mathbf{v}_1+\mathbf{v}_2\big)
-
\omega_m\left(\mathbf{v}_1,\mathbf{v}_2\right)\right]\right.\\
&&\left.-\frac 12\,\left\|\left(\mathbf{v}_1-\mathbf{v}_2\right)
-\frac{\theta_2-\theta_1}{\Phi(m)}\, \xi_M(m)\right\|^2\right\}.
\end{eqnarray*}
\end{defn}

\begin{rem}
If $(\theta_2-\theta_1)\,\xi_M(m)=\mathbf{0}$, then $E(\upsilon_1,\upsilon_2)=\psi_2(\mathbf{v}_1,\mathbf{v}_2)/\Phi(m)$,
where $\psi_2$ is the invariant introduced in \cite{bsz} and \S 3 of \cite{sz} to describe the universality of
the leading scaling asymptotics of Szeg\"{o} kernels.

We have $\Re \big(E(\upsilon_1,\upsilon_2)\big)\le 0$ for any $\upsilon_j$, and
$\Re \big(E(\upsilon_1,\upsilon_2)\big)= 0$ if and only if $\upsilon_1-\upsilon_2\in \mathrm{span}\big(\xi_X(x)\big)$.
\end{rem}

\begin{thm}
\label{thm:circle-case}
Let $M$ be a connected d-dimensional complex projective manifold, and $(A,h)$ an Hermitian ample line bundle on it;
suppose that the curvature of the unique compatible connection is $\Theta=-2i\omega$, where $\omega$ is
K\"{a}hler.
Let $\mu:\mathbb{T}^1\times M\rightarrow M$ be an holomorphic Hamiltonian action on $(M,2\omega)$ with moment map
$\Phi>0$, admitting the linearization $\widetilde{\mu}:\mathbb{T}^1\times A\rightarrow A$; assume that
$h$ is $\widetilde{\mu}$-invariant.
Then:

\begin{enumerate}
  \item $\widetilde{\Pi}_k=0$ for any $k\le 0$.
  \item For any $\epsilon,\,C>0$, uniformly for
  $\mathrm{dist}_X\left(\mathbb{T}^1\cdot x,\mathbb{T}^1\cdot y\right)\ge Ck^{\epsilon-\frac 12}$ we have
  $\widetilde{\Pi}_k(x,y)=O\left(k^{-\infty}\right)$ as $k\rightarrow +\infty$; here $\mathbb{T}^1\cdot x$
  is the $\widetilde{\mu}$-orbit of $x\in X$.
  \item Uniformly in $x\in X$ and
in
$\upsilon_l=(\theta_l,\mathbf{v}_l)\in T_xX\cong\mathbb{R}\times T_mM$
satisfying
\begin{center}
$\upsilon_l\in \xi_X(x)^\perp$ and
$\big\|\upsilon_l\big\|\le C\,k^{1/9}$,
\end{center}
as $k\rightarrow +\infty$ we have
\begin{eqnarray*}
\lefteqn{\widetilde{\Pi}_k\left(x+\frac{\upsilon_1}{\sqrt{k}},
x+\frac{\upsilon_2}{\sqrt{k}}\right)\sim \left(\frac k\pi\right)^\mathrm{d}\,\Phi(m)^{-(\mathrm{d}+1)}
\,e^{i\sqrt{k}\,(\theta_1-\theta_2)/\Phi(m)}}\\
&&\cdot
\left(\sum_{t\in T_m}t^{k}\,e^{E\big(d_x\widetilde{\mu}_{t^{-1}}(\upsilon_1),\upsilon_2\big)}\right)\cdot
\left(1+\sum _{j\ge 1}R_j(m,\upsilon_1,\upsilon_2)\,k^{-j/2}\right),
\end{eqnarray*}
for certain smooth functions $R_j$, polynomal in the $\upsilon_l$'s.
\end{enumerate}
\end{thm}

The condition $(\theta_l,\mathbf{v}_l)\in \xi_X(x)^\perp$ could be replaced by the condition
$(\theta_l,\mathbf{v}_l)\in W$, where $W\subset T_xX$ is any fixed vector subspace with
$\xi_X(x)\not\in W$.

Actually, if $\upsilon_1=\upsilon_2=\mathbf{0}$ an asymptotic expansion in descending powers
of $k$ holds (rather than $k^{1/2}$); this can be seen either by modifying the proof of Theorem \ref{thm:circle-case}, or else by
applying Theorem \ref{thm:higher-dim-case} below. We state this as a

\begin{cor}
\label{cor:point-wise-circle}
Under the assumptions of Theorem \ref{thm:circle-case}, for any $x\in X$ the following
asymptotic expansion holds as $k\rightarrow +\infty$:
\begin{equation*}
\widetilde{\Pi}_k(x,x)\sim \left(\frac k\pi\right)^{\mathrm{d}}\cdot \Phi(m)^{-(\mathrm{d}+1)}\,
\sum _{g\in T_m}g^k \cdot \left(1+\sum_{l\ge 1}k^{-l}B_l(m)\right),
\end{equation*}
where $m=\pi(x)$, and each $B_l$ is a smooth function on $M$.
\end{cor}

In particular, $\widetilde{\Pi}_k(x,x)=0$ unless $k$ is a multiple of $|T_m|$.
Now the cardinality $|T_m|$ needn't be constant on $M$, but it does attain a generic minimal value $\ell$
on some dense open subset $M^0\subseteq M$ (Corollary B47 of \cite{ggk}). In fact, in the present Abelian
setting there is a finite subgroup $L\subseteq \mathbb{T}^1$ which is the stabilizer of a general $x\in X$,
and $\ell=|L|$
(we might as well quotient by $L$ and reduce to the case where it is trivial, whence $\ell=1$). Clearly,
$\widetilde{H}_{k}(X)=\{0\}$ unless $\ell|k$. On the other hand, we have the following:

\begin{cor}
\label{cor:dim-circle-case}
In the hypothesis of Theorem \ref{thm:circle-case},
$$
\lim_{k\rightarrow +\infty}\frac{\mathrm{d}!}{(\ell k)^\mathrm{d}}\,\dim \left(\widetilde{H}_{kl}(X)\right)=
\ell\,\int_M\Phi^{-(d+1)}\,c_1(A)^\mathrm{d}.
$$
\end{cor}

We shall now consider point-wise expansions and scaling asymptotics
for general $\mathrm{g}$.

\begin{defn}
If $\varpi=(\varpi_1,\ldots,\varpi_\mathrm{g})\in \mathbb{Z}^\mathrm{g}$ and $t=(t_1,\ldots,t_{\mathrm{g}})\in\mathbb{T}^\mathrm{g}$,
we shall set $t^\varpi=:t_1^{\varpi_1}\cdots t_\mathrm{g}^{\varpi_\mathrm{g}}$ and $\chi_\varpi(t)=:t^\varpi$.
\end{defn}

Thus
$\widetilde{\Pi}_\varpi$ is the Schwartz kernel of the orthogonal projector of $L^2(X)$ onto the subspace
$$
H_\varpi(X)=:\Big\{s\in H(X)\,:\,s\left(\widetilde{\mu}_{t^{-1}}(x)\right)
=t^\varpi\,s(x)\,\forall\,x\in X,\,t\in\mathbb{T}^\mathrm{g}\Big\}.
$$

Given that $\mathbf{0}\not\in \Phi(M)$,
if in addition $\Phi$ is transverse to $\mathbb{R}_+\cdot \varpi$ then $M_\varpi=:\Phi^{-1}(\mathbb{R}_+\cdot \varpi)$
(if non-empty)
is a connected compact submanifold of $M$, of real codimension $\mathrm{g}-1$ (Lemma \ref{lem:connected}).
Furthermore, if $m\in M_\varpi$ and $m=\pi(x)$, then the stabilizer subgroup $T_m\subseteq \mathbb{T}^\mathrm{g}$
of $x$ for $\widetilde{\mu}$
is finite (Lemma \ref{lem:moment-map-transverse}).

Under the same transversality assumption, furthermore,
the normal bundle $N$ of $M_\varpi$ in $M$ is naturally isomorphic to the
vector bundle with fiber $\ker \big(\Phi(m)\big)$ ($m\in M_\varpi$) (Lemma \ref{lem:N-ker-Phi-transverse}).
Thus for every
$m\in M_\varpi$ we have two Euclidean structures on $\ker\big(\Phi(m)\big)$, induced from
$\mathfrak{t}$ and $T_mM$, respectively. Let $D(m)$ be the matrix representing the
latter Euclidean product on $N_m$, with respect to an orthonormal basis of the former. Then
$\det D(m)$ is independent of the choice of an orthonormal basis for $\ker \big(\Phi(m)\big)\subseteq \mathfrak{t}$,
and so it determines a positive smooth function on $M_\varpi$.

\begin{defn}
\label{defn:calD}
Define $\mathcal{D}\in \mathcal{C}^\infty(M_\varpi)$ by setting:
$$
\mathcal{D}(m)=:\sqrt{\det D(m)}\,\,\,\,\,\,\,\,\,\,\,\,\,\,\,(m\in M_\varpi).
$$
\end{defn}

\begin{thm}
\label{thm:higher-dim-case}
Let $\mu:\mathbb{T}^\mathrm{g}\times M\rightarrow M$ be a holomorphic Hamiltonian action on $(M,2\omega)$, with
moment map $\Phi:M\rightarrow \mathfrak{t}^\vee$ satisfying $\mathbf{0}\not\in \Phi(M)$.
Suppose that $\widetilde{\mu}:\mathbb{T}^\mathrm{g}\times A\rightarrow A$ is a linearization of $\mu$ leaving $h$
invariant. Then for any $\varpi\in \mathbb{Z}^\mathrm{g}$ the following holds.

If  $m=\pi(x)$ and $\Phi (m)\not\in \mathbb{R}_+\cdot \varpi$, then $\widetilde{\Pi}_{k\varpi}(x,x)=O\left(k^{-\infty}\right)$ as
$k\rightarrow +\infty$.

Assume that $\Phi$ is transversal
to $\mathbb{R}_+\cdot \varpi$. Then for every $m=\pi(x)\in M_\varpi$ as $k\rightarrow +\infty$ we have
\begin{eqnarray*}
\widetilde{\Pi}_{k\varpi}(x,x)&\sim&\frac{1}{\left(\sqrt{2}\pi\right)^{\mathrm{g}-1}}
\left(\|\varpi\|\cdot \frac k\pi\right)^{\mathrm{d}+(1-\mathrm{g})/2}\cdot
\sum_{g\in T_m}\chi_\varpi(g)^k\\
&&\cdot \frac{1}{\mathcal{D}(m)}\,\left(\frac{1}{\|\Phi(m)\|}\right)^{d+1+(1-\mathrm{g})/2}
\cdot \left(1+\sum_{l\ge 1}B_l\,k^{-l}\right),
\end{eqnarray*}
where the $B_l$'s are smooth functions on $M_\varpi$.
\end{thm}

If non-empty, the locus $M'\subseteq M$ where $\mu$ is locally free is open and dense
(Corollary B.47 of \cite{ggk}). Similarly,
if $M_\varpi'=:M_\varpi\cap M'\neq \emptyset$, then $M_\varpi'$ is open and dense in $M_\varpi$.
In this case, the leading term in the asymptotic expansion in
Theorem \ref{thm:higher-dim-case} may be given an alternative expression if
$m\in M'_\varpi$.

If $\eta\in \mathfrak{t}$, let
$\eta_M$ be the induced smooth vector field on $M$; for any $m\in M$, evaluation yields a linear map $\mathrm{val}_m:\mathfrak{t}\rightarrow T_mM$, $\xi\mapsto \xi_M(m)$. Clearly, $m\in M'$ if and only if
$\mathrm{val}_m$ is injective.

For $m\in M'$, let
$\langle \,,\,\rangle _m$ be the Euclidean product on $\mathfrak{t}$ given by pull-back under $\mathrm{val}_m$
of the Riemannian structure $g$ of $T_mM$:
$$
\langle \xi,\upsilon\rangle _m=:g_m\big(\xi_M(m),\upsilon_M(m)\big) \,\,\,\,\,\,\,\,\,\,(\xi,\upsilon \in \mathfrak{t}).
$$
Also, let
$\|\cdot\|_m:\mathfrak{t}\rightarrow \mathbb{R}$ be the corresponding norm. The same notation will denote the
Euclidean product and the norm induced on $\mathfrak{t}^*$ under duality.

Let
$V_{\mathrm{eff}}:M'\rightarrow \mathbb{R}$ be the effective potential of $\mu$, that is, $V_{\mathrm{eff}}(m)$
is the volume of the orbit $\mathbb{T}\cdot m$ ($m\in M'$).
If  $m=\pi(x)\in M'_\varpi$ and $\Phi (m)\in \mathbb{R}_+\cdot \varpi$, then
\begin{eqnarray}
\label{eqn:alternative-thm-2}
\widetilde{\Pi}_{k\varpi}(x,x)&\sim&\left(\|\varpi\|\cdot \frac k\pi\right)^{\mathrm{d}+(1-\mathrm{g})/2}\cdot
\frac{1}{|T_m|}\,\sum_{g\in T_m}\chi_\varpi(g)^k\\
&&\cdot \frac{2^{(\mathrm{g}+1)/2}\pi}{V_{\mathrm{eff}}(m)\,\|\Phi(m)\|_m}\,\left(\frac{1}{\|\Phi(m)\|}\right)^{d+(1-\mathrm{g})/2}
\cdot \left(1+\sum_{l\ge 1}B_l\,k^{-l}\right).\nonumber
\end{eqnarray}
For $\mathrm{g}=1$, one can see that
$\|\Phi(m)\|=(2\pi)^{-1}\cdot V_{\mathrm{eff}}(m)\,|T_m|\,\|\Phi(m)\|_m$, so that
(\ref{eqn:alternative-thm-2}) tallies with Theorem \ref{thm:circle-case}.

Let us make more precise the sense in which $\widetilde{\Pi}_{k\varpi}$ localizes around
$X_\varpi$. First we have:

\begin{thm}
\label{thm:scaling-limit-general}
Assume that $\mathbf{0}\not\in \Phi(M)$ and $\Phi$ is transversal
to $\mathbb{R}_+\cdot \varpi$. Let $C,\epsilon>0$. Then, uniformly for
$$
\max\big\{\mathrm{dist}_X\big(\mathbb{T}\cdot x,X_\varpi\big),
\mathrm{dist}_X\big(\mathbb{T}\cdot x,\mathbb{T}\cdot y\big)\big\}
\ge C\,k^{\epsilon-1/2},
$$
we have
$\widetilde{\Pi}_{k\varpi}(x,y)=O\left(k^{-\infty}\right)$ as
$k\rightarrow +\infty$.
\end{thm}

Given this, we are led to studying scaling asymptotics for $\widetilde{\Pi}_{k\varpi}$ at points of
$M_\varpi$.
If $\mathbf{v}\in T_mM$ we shall write $x+\mathbf{v}=x+(0,\mathbf{v})$.

\begin{defn}
For $m\in M$ and $\mathbf{v}_1,\mathbf{v}_2\in T_mM$, let us set
$$
\lambda_\varpi(m)=:\|\varpi\|/\|\Phi(m)\|
$$
and
$$
H_m(\mathbf{v}_1,\mathbf{v}_2)=:\lambda_\varpi(m)\,\Big[-i\,\omega_m(\mathbf{v}_1,\mathbf{v}_2)-
\big(\|\mathbf{v}_1\|^2+\|\mathbf{v}_2\|^2\big)\Big].
$$
\end{defn}

\begin{thm}
\label{thm:scaling-limit-general-precise}
Assume that $\mathbf{0}\not\in \Phi(M)$, $M_\varpi\neq \emptyset$ and $\Phi$ is transversal
to $\mathbb{R}_+\cdot \varpi$. Then:

\begin{enumerate}
  \item $\widetilde{\Pi}_{k\varpi}=0$ for any $k\le 0$.
 \item Uniformly in $x\in X_\varpi$ and
in
$\mathbf{v}_l\in  T_mM$, with $m=\pi(x)$,
satisfying
\begin{center}
$\mathbf{v}_l\in N_m$ and
$\big\|\mathbf{v}_l\big\|\le C\,k^{1/9}$,
\end{center}
as $k\rightarrow +\infty$ we have
\begin{eqnarray*}
\lefteqn{\widetilde{\Pi}_{k\varpi}\left(x+\frac{\mathbf{v}_1}{\sqrt{k}},
x+\frac{\mathbf{v}_2}{\sqrt{k}}\right)}\\
&\sim &\frac{1}{(\sqrt{2}\pi)^{\mathrm{g}-1}}\left(\|\varpi\|\cdot\frac k\pi\right)^{\mathrm{d}+(1-\mathrm{g})/2}\,
\left(\sum_{t\in T_m}\chi_\varpi(t)^k\,e^{H_m\big(d_m\mu_{t^{-1}}(\mathbf{v}_1),\mathbf{v}_2\big)}\right)\\
&&\cdot \frac{1}{\mathcal{D}(m)}\,
\left(\frac{1}{\|\Phi(m)\|}\right)^{\mathrm{d}+1+(1-\mathrm{g})/2}\cdot
\left(1+\sum _{j\ge 1}R_j(m,\mathbf{v}_1,\mathbf{v}_2)\,k^{-j/2}\right),
\end{eqnarray*}
for certain smooth functions $R_j$, polynomal in the $\mathbf{v}_l$'s.
\end{enumerate}
\end{thm}

More generally for any $\xi\in \mathfrak{t}$, let $\xi_X$ be smooth vector field on $X$ induced by the
infinitesimal action of $\xi$. For any $x\in X$, let
$$
\mathfrak{t}_X(x)=:\big\{\xi_X(x)\,:\,\xi\in \mathfrak{t}\big\}\subseteq T_xX.
$$
Thus $\mathfrak{t}_X(x)\subseteq T_xX$ is a vector subspace of $T_xX$, of dimension
g if $\Phi$ is transverse to $\mathbb{R}_+\cdot \varpi$ and $x\in X_\varpi$
(Lemma \ref{lem:moment-map-transverse}).
Then it follows from the proof of Theorem \ref{thm:scaling-limit-general-precise}
that similar expansions hold for rescaled displacements
$x+\upsilon_j/\sqrt{k}$, where $\upsilon_j\in T_xX$ satisfy, say, $\big\|\mathbf{v}_j\big\|\le C\,k^{1/9}$,
$\upsilon_j\in \mathfrak{t}_X(x)^\perp$ (the expression for the leading term will change).

Under the same hypothesis, we can
estimate the asymptotic growth of $\dim \widetilde{H}_{k\varpi}(X)$ as $k\rightarrow +\infty$.
The stabilizer subgroup of
any $x\in X_\varpi=:\pi^{-1}\left(M_\varpi\right)$ for $\widetilde{\mu}$ is finite, and
on a dense open subset of
$X_\varpi$ it is constant. Let then
$L_\varpi\subseteq \mathbb{T}^\mathrm{g}$ be the stabilizer subgroup of a general $x\in X_\varpi$.
Although the proof of Corollary \ref{cor:asymp-bound-dim-gen-case} below works with minor modifications in
the general case, for the sake of brevity let us restrict ourselves to the special case where $L_\varpi$ is trivial,
that is, $\widetilde{\mu}$ is generically free on $X_\varpi$. Then

\begin{cor}
\label{cor:asymp-bound-dim-gen-case}
Under the hypothesis of Theorem \ref{thm:scaling-limit-general-precise}, assume in addition that
$L_\varpi$ is trivial. Then
\begin{eqnarray*}
\lefteqn{\lim_{k\rightarrow +\infty}\left(\|\varpi\| \,\frac k\pi\right)^{-(\mathrm{d}+1-\mathrm{g})}
\,\dim\left(\widetilde{H}_{k\varpi }(X)\right)}\\
&=&
\frac{1}{(2\pi)^{\mathrm{g}-1}}\,\int_{M_\varpi}\|\Phi(m)\|^{-(\mathrm{d}+2-\mathrm{g})}\cdot
\frac{1}{\mathcal{D}(m)}\, dV_{M_\varpi}(m).
\end{eqnarray*}
\end{cor}

While we have restricted the exposition to
the complex projective setting, the results in this paper admit natural
generalizations to the almost K\"{a}hler context, following the theory of generalized Szeg\"{o}
kernels in \cite{bg} and \cite{sz}.

In closing, we remark that in recent years the local asymptotics of equivariant components of
Bergman-Szeg\"{o} kernels weighted by toric actions have been studied by several authors (see for example
\cite{hsb}, \cite{bgw}, \cite{stz}).
There are several deep variants in these asymptotics, but the emphasis has been
on working at a level $k$ tending to infinity of the standard circle action, so that in fact
one splits
$H^0\left(M,A^{\otimes k}\right)$ over the irreducibles of the group.
The present point of view is different, inasmuch as the additional
symmetry is considered \textit{per se}, on the same footing as the standard circle action in
the classical TYZ expansion. As exhibited by the previous examples and statements,
this accounts for some sharp differences in the
asymptotic concentrations of the projection kernels, as regards both the rate of growth and
the geometric loci involved. For instance, as $k\rightarrow +\infty$
the \lq non-standard
coherent states\rq\, $\widetilde{\Pi}_{k\varpi}(\cdot,x)$ will concentrate on
the $\mathbb{T}^\mathrm{g}$-orbit of $x$ under $\widetilde{\mu}$, rather
than on the inverse image of the underlying orbit in $M$.

\bigskip

\textbf{Acknowledgments.} I am grateful to the referee for several valuable comments and
for suggesting various improvements in
presentation.

\section{Preliminaries}

\subsection{The equivariant spaces}
\begin{lem}
\label{lem:fd-smooth}
If $\mathbf{0}\not\in \Phi(M)$, then $\widetilde{H}_\varpi(X)$ is finite-dimensional for every $\varpi\in \mathbb{Z}^\mathrm{g}$.
Furthermore, the smooth function $x\mapsto\widetilde{\Pi}_\varpi(x,x)$ ($x\in X$) descends to a smooth function on $M$.
\end{lem}

\begin{proof}
Since the standard circle action on $X$ commutes with $\widetilde{\mu}$, it leaves $\widetilde{H}_\varpi(X)$ invariant
and so
$$
\widetilde{H}_\varpi(X)=\bigoplus_{k\ge 0}\widetilde{H}_\varpi(X)\cap H_k(X).
$$
Since each $H_k(X)$ is finite-dimensional, it suffices to show that for any $\varpi$ one has
$\widetilde{H}_\varpi(X)\cap H_k(X)\neq \{0\}$ for at most finitely many $k$'s. By the theory of \cite{gs1},
one has $\widetilde{H}_\varpi(X)\cap H_k(X)= \{0\}$ unless $\varpi\in k\cdot \Phi(M)$. Therefore, if $a=:\min\|\Phi\|$,
$A=:\max\|\Phi\|$, then $0<a\le A$, and $\widetilde{H}_\varpi(X)\cap H_k(X)= \{0\}$ unless $ka\le \|\varpi\|\le kA$, that is,
unless $\|\varpi\|/A\le k\le \|\varpi\|/a$.

Now an orthonormal basis $(s_j)$ of $\widetilde{H}_\varpi(X)$ can be built by taking the union of orthonormal basis of
$\widetilde{H}_\varpi(X)\cap H_k(X)$ for each $k$ for which the latter intersection is non-empty. If $s\in H_k(X)$,
on the other hand,
then clearly $s\left(e^{i\theta}\cdot x\right)\cdot \overline{s\left(e^{i\theta}\cdot y\right)}=
s(x)\cdot \overline{s\left(y\right)}$ for any $x,y\in X$, where $e^{i\theta}\cdot x$ denotes the standard circle action.
Hence $x\mapsto \big|s(x)\big|^2$ descends to a smooth function on $M$.
Since
$$
\widetilde{\Pi}_\varpi(x,x)=\sum_j \big|s_j(x)\big|^2
$$
and each of the finitely many summands descends to a smooth function on $M$, the same is true of $\widetilde{\Pi}_\varpi(x,x)$.
\end{proof}

\begin{lem}
\label{lem:only-pos-k-matter}
If $\mathbf{0}\not\in \Phi(M)$ and $\Phi^{-1}(\mathbb{R}_+\cdot \varpi)\neq \emptyset$,
then
$\widetilde{H}_{k\varpi}(X)=\{0\}$ for every $k\le 0$.
\end{lem}

\begin{proof}
Suppose to the contrary that $\widetilde{H}_{k\varpi}(X)$ is non-zero for some $k\le 0$.
Then there exits $l>0$ such that $\widetilde{H}_{k\varpi}(X)\cap H_l(X)\neq \{0\}$,
and therefore $k\varpi\in l\Phi(M)$. Thus, $(k/l)\,\varpi\in \Phi(M)$, and
on the other hand $\lambda\,\varpi\in \Phi(M)$ for some $\lambda>0$.
Since $\Phi(M)$ is convex \cite{gs0}, this forces $\mathbf{0}\in \Phi(M)$ and therefore
a contradiction.
\end{proof}

\subsection{The geometric setting}
\begin{lem}
\label{lem:connected}
$M_\varpi=\Phi^{-1}(\mathbb{R}_+\cdot \varpi)$ is connected for any $\varpi$.
\end{lem}

\begin{proof}
Since $\mathbf{0}\not\in \Phi(M)$ by assumption, $M_\varpi$ is closed in $M$, hence compact.
Also, since $\Phi(M)\subseteq \mathfrak{t}^*\cong \mathbb{R}^\mathrm{g}$ is convex, if $M_\varpi\neq \emptyset$
then
$$
\Phi(M_\varpi)=\Phi(M)\cap \big(\mathbb{R}_+\cdot \varpi\big)=[a\varpi,b\varpi]
$$
for some $0<a<b$.
Let $A,B\subseteq M_\varpi$ be non-empty disjoint closed subsets such that $M_\varpi=A\cup B$.
Then $\Phi(A),\,\Phi(B)$ are non-empty closed subsets of $[a\varpi,b\varpi]$, and
$\Phi(A)\cup \Phi(B)=[a\varpi,b\varpi]$. Hence there is $c\in [a,b]$ such that
$c\varpi\in \Phi(A)\cap \Phi(B)$. Thus, $A\cap \Phi^{-1}(c\varpi)\neq\emptyset$,
$B\cap \Phi^{-1}(c\varpi)\neq\emptyset$. Therefore,
$$
\Phi^{-1}(c\varpi)=\left(\Phi^{-1}(c\varpi)\cap A\right)\cup \left(\Phi^{-1}(c\varpi)\cap B\right)
$$
is a union of two non-empty disjoint closed subsets. This is absurd, since each level set of $\Phi$
is connected \cite{l}.
\end{proof}

For $m\in M'$, consider the Euclidean vector space
$\mathfrak{t}_m=\big(\mathfrak{t},\langle\,,\,\rangle _m\big)$, where $\langle\,,\,\rangle_m$ is induced by pull-back
under the injective linear map $\mathrm{val}_m:\mathfrak{t}\rightarrow T_mM$; we denote by $\mathfrak{t}^*_m=:\big(\mathfrak{t}^*,\langle\,,\,\rangle _m\big)$ the Euclidean structure induced
on $\mathfrak{t}^*$ under duality.
Also, we set $\mathfrak{t}_M(m)=:\mathrm{val}_m\big(\mathfrak{t}\big)\subseteq T_mM$.
Thus $\mathfrak{t}_M$ is the rank-g vector sub-bundle of $\left.TM\right|_{M'}$
generated by the vector fields $\xi_M$ ($\xi\in\mathfrak{t}$). Obviously $\mathrm{val}_m$
is an isometry $\mathfrak{t}_m\cong \mathfrak{t}_M(m)$.

Let $J:TM\rightarrow TM$ be the complex structure. For any $m\in M'$, the Riemannian orthocomplement to
the fiber of $\Phi$ through $m$ is
\begin{equation}
\label{eqn:riemannian-orth-phi}
T_m\Big(\Phi^{-1}\big(\Phi(m)\big)\Big)^\perp =J_m\big(\mathfrak{t}_M(m)\big),
\end{equation}
and $d_m\Phi$ induces by restriction a linear isomorphism $\varsigma_m:J_m\big(\mathfrak{t}_M(m)\big)\rightarrow \mathfrak{t}^*$
(a rephrasing of the metric pairing $\langle\,,\,\rangle_m$).

\begin{lem}
\label{lem:varsigma-isometry}
$\varsigma_m$ is an isometry $J_m\big(\mathfrak{t}_M(m)\big)\cong \mathfrak{t}^*_m$.
\end{lem}

\begin{proof}
Suppose $f,g\in \mathfrak{t}^*$ and $\xi,\eta\in \mathfrak{t}$ are such that
$f(\mu)=\langle\mu,\xi\rangle_m$,
$g(\mu)=\langle\mu,\eta\rangle_m$ ($\mu\in \mathfrak{t}$).
Then by definition $\langle f,g\rangle_m=\langle\xi,\eta\rangle_m$.

On the other hand, it is readily seen that $f=\varsigma_m\left(J_m\left(\xi_M(m)\right)\right)$, $g=\varsigma_m\left(J_m\left(\eta_M(m)\right)\right)$.
Since $J_m$ is orthogonal,
$$
g_m\Big(J_m\big(\xi_M(m)\big),J_m\big(\eta_M(m)\Big)=g_m\Big(\xi_M(m),\eta_M(m)\Big)=
\langle\xi,\eta\rangle_m.
$$
\end{proof}

Now suppose $m\in M'_\varpi$ ($\varpi\neq \mathbf{0}$), so that $\Phi(m)=\lambda\,\varpi$
for some $\lambda>0$. Consider
the unique $\xi=\xi(m,\varpi)\in \mathfrak{t}$ such that $d_m\Phi\Big(J_m\big(\xi_M(m)\big)\Big)=\varpi$; then
\begin{equation}
\label{eqn:riemannian-orth-phi-hom}
T_m\Big(\Phi^{-1}\big(\mathbb{R}_+\cdot \varpi\big)\Big) =T_m\left(\Phi^{-1}\big(\lambda\,\varpi\big)\right)
\oplus \mathrm{span}\Big\{J_m\big(\xi_M(m)\big)\Big\}.
\end{equation}
It follows from (\ref{eqn:riemannian-orth-phi}) and (\ref{eqn:riemannian-orth-phi-hom}) that
\begin{equation}
\label{eqn:riemannian-orth-phi-hom-1}
T_m\Big(\Phi^{-1}\big(\mathbb{R}_+\cdot \varpi\big)\Big)^\perp =J_m\big(\mathfrak{t}_M(m)\big)
\cap \mathrm{span}\Big\{J_m\big(\xi_M(m)\big)\Big\}^\perp.
\end{equation}

The left-hand side of (\ref{eqn:riemannian-orth-phi-hom-1}) is the fiber at $m$ of the normal
bundle of $M_\varpi\cap M'$ in $M'$, $N_m'$. We then have an orthogonal direct sum
decomposition
\begin{equation}
\label{eqn:orthogonal-Tm}
J_m\big(\mathfrak{t}_M(m)\big)=N_m'\oplus \mathrm{span}\Big\{J_m\big(\xi_M(m)\big)\Big\}.
\end{equation}

\begin{lem}
\label{lem:orthogonal-direct-t-M}
Set $F(m)=:d_m\Phi\big(N_m'\big)$. Then there is an orthogonal direct sum
$$
\mathfrak{t}^*_m= F(m)\oplus \mathrm{span}\{\varpi\}= F(m)\oplus \mathrm{span}\{\Phi(m)\}.
$$
\end{lem}

\begin{proof}
Just apply the isometry $\varsigma_m$ to (\ref{eqn:orthogonal-Tm}).
\end{proof}

We can give the following alternative description of $N'$.

\begin{lem}
\label{lem:N-ker-Phi}
For any $m\in M'_\varpi$, we have
$N'_m=J_m\circ \mathrm{val}_m\big(\ker \Phi(m)\big)$.
\end{lem}

\begin{proof}
Let $\xi\in \mathfrak{t}$ be such that
$d_m\Phi\Big(J_m\big(\xi_M(m)\big)\Big)=\varpi$. If $\eta\in \mathfrak{t}$, then
\begin{eqnarray*}
\lefteqn{J_m\big(\eta_M(m)\big)\in N_m'\,\Leftrightarrow\,
g_m\Big(J_m\big(\eta_M(m)\big),J_m\big(\xi_M(m)\big)\Big)=0}\\
&\Leftrightarrow&g_m\big(\eta_M(m),\xi_M(m)\big)=0\,\Leftrightarrow\,
\omega_m\Big(\eta_M(m),J_m\big(\xi_M(m)\big)\Big)=0\\
&\Leftrightarrow&
d_m\Phi_\eta\Big(J_m\big(\xi_M(m)\big)\Big)=0\,\Leftrightarrow\,\left< d_m\Phi\Big(J_m\big(\xi_M(m)\Big),\eta\right>=0\\
&\Leftrightarrow&\eta\in \ker\left(d_m\Phi\Big(J_m\big(\xi_M(m)\Big)\right)\,\Leftrightarrow\,
\eta\in \ker(\varpi)=\ker\big(\Phi(m)\big),
\end{eqnarray*}
where in the last equality we have used that $\Phi(m)=\lambda\,\varpi$ for some $\lambda>0$.
\end{proof}

Lemma \ref{lem:N-ker-Phi} establishes a natural isomorphism between $N'$ and the restriction to $M'_\varpi$
of the globally defined vector bundle $V$ on $M$ given by $V(m)=:\ker\Phi(m)$.

Let us now dwell on the hypothesis of Theorem \ref{thm:higher-dim-case}. To this end, let us introduce
the closed symplectic cone in $T^*X\setminus\{0\}$ sprayed by the connection 1-form:
$$
\Sigma=:\big\{(x,r\alpha_x)\,:\,x\in X,r>0\big\}.
$$
This cone is crucial in the microlocal description of the Szeg\"{o} kernel as an FIO \cite{bs}
and in the theory of Toeplitz operators \cite{bg};
in particular, the wave front of $\Pi$ is the anti-diagonal
$$
\Sigma^\sharp=:\big\{(x,r\alpha_x,x,-r\alpha_x)\,:\,x\in X,r>0\big\}\subseteq T^*(X)\times T^*(X).
$$
Let $\omega_\Sigma$ be the restriction to $\Sigma$ of the symplectic structure of
$T^*X$.
Then $\Sigma\cong X\times \mathbb{R}_+\cong A^\vee\setminus\{0\}$ as manifolds,
and $\omega_\Sigma$ is as follows. Let $r$ be the cone coordinate
on $\Sigma$ and $\theta$ be the \lq circle\rq \,coordinate on $X$, locally defined,
and pulled-back to $\Sigma$. Then $\omega_\Sigma=2r\,\omega+dr\wedge d\theta$, where
$\omega$ is pulled-back from $M$ (symbols of pull-back are omitted).

In general,
the Hamiltonian vector field $\upsilon_f$ on $(M,2\omega)$ of any real $f\in \mathcal{C}^\infty(M)$
lifts to the contact vector field
$\widetilde{\upsilon}_f=\upsilon^\sharp_f-f\,(\partial/\partial\theta)$
on $(X,\alpha)$; here $\upsilon^\sharp_f$ is the horizontal lift and
$\partial/\partial \theta$ is the generator of the standard circle action.
The cotangent lift of the contact flow of $\phi^X_\tau:X\rightarrow X$ ($\tau\in \mathbb{R}$)
of $\widetilde{\upsilon}_f$ is a Hamiltonian flow on $T^*X$, which leaves $\Sigma$ invariant.
In fact, its restriction to $\Sigma$ is
$\Phi^\Sigma_\tau(x,r\,\alpha_x)=:\left(\phi^X_\tau(x),r\,\alpha_{\phi^X_\tau(x)}\right)$, and this
is the Hamiltonian flow of $rf$ on $(\Sigma,\omega_\Sigma)$.
In particular, since the action of $\mathbb{T}^\mathrm{g}$ on $X$ preserves $\alpha$,
$\widetilde{\mu}$ lifts to an Hamiltonian action
on $\Sigma$; the moment map of the latter is
$\widetilde{\Phi}(x,r\alpha_x)=r\,\Phi\big(\pi(x)\big)$.

\begin{lem}
\label{lem:moment-map-transverse}
$\Phi$ is transverse to $\mathbb{R}_+\cdot \varpi$ if and only if $\varpi$ is a regular value of
$\widetilde{\Phi}$.
\end{lem}

\begin{proof}
We have $(x,r\alpha_x)\in \widetilde{\Phi}^{-1}(\varpi)$ if and only if $m\in M_\varpi\Phi^{-1}(\mathbb{R}_+\cdot \varpi)$
and $r=\|\varpi\|/\|\Phi(m)\|$, where $m=:\pi(x)$. Hence $\widetilde{\Phi}^{-1}(\varpi)$ is an $S^1$-bundle over $\Phi^{-1}(\mathbb{R}_+\cdot \varpi)$.
The statement follows since for any $(x,r\alpha_x)\in \Sigma$ with $m=\pi(x)$ we have
\begin{equation}
\label{eqn:transverse-conic-mm}
d_{(x,r\alpha_x)}\widetilde{\Phi}\big(T_{(x,r\alpha_x)}\Sigma\big)=d_{m}\Phi(T_mM)+\mathrm{span}\{\Phi(m)\}.
\end{equation}
\end{proof}

In particular, if $\Phi$ is transverse to $\mathbb{R}_+\cdot \varpi$ then $\widetilde{\mu}$ is locally
free on $X_\varpi=:\pi^{-1}(M_\varpi)$; therefore, the stabilizer subgroup $T_m\subseteq \mathbb{T}^\mathrm{g}$
of any $x\in X_\varpi$ is finite, and
depends only on $m=\pi(x)$.

Inspection of (\ref{eqn:transverse-conic-mm}) immediately yields:

\begin{lem}
\label{lem:transverse-conic-mm}
$\Phi$ is transverse to $\mathbb{R}_+\cdot \varpi$ if and only if for every $m\in M_\varpi$ either
$\mathrm{rank}(d_m\Phi)=\mathrm{g}$, or else $\mathrm{rank}(d_m\Phi)=\mathrm{g}-1$ and $\Phi(m)\not\in d_m\Phi(T_mM)$.
\end{lem}

\begin{cor}
\label{cor:transverse-conic-mm}
$\Phi$ is transverse to $\mathbb{R}_+\cdot \varpi$ if and only if the following two conditions hold:
\begin{enumerate}
  \item $\mathrm{rank}(\mathrm{val}_m)\ge \mathrm{g}-1$ for every $m\in M_\varpi$;
  \item if $m\in M_\varpi$ and $\mathrm{rank}(\mathrm{val}_m)= \mathrm{g}-1$ (that is,
  $m\in M_\varpi\setminus M_\varpi'$), and if $\ker(\mathrm{val}_m)=\mathrm{span}\{\xi\}$,
  then $\langle \Phi(m),\xi\rangle \neq 0$.
\end{enumerate}
\end{cor}

In other words, if $\Phi$ is transverse to $\mathbb{R}_+\cdot \varpi$ if and only if
$\mathrm{val}_m$ induces by restriction an injective linear map $\ker\big(\Phi(m)\big)\rightarrow T_mM$
for every $m\in T_mM$. The image of this map as $m\in M_\varpi$ varies forms a vector bundle, which is naturally
isomorphic to the normal bundle, as we now show.

Assuming $\Phi$ is transverse to $\mathbb{R}_+\cdot \varpi$, let $N$ be the
normal bundle of $M_\varpi$ in $M$; clearly, $N$ restricts $N'$ on $M_\varpi'$.
We can extend Lemma \ref{lem:N-ker-Phi} as follows:

\begin{lem}
\label{lem:N-ker-Phi-transverse}
If $\Phi$ is transverse to $\mathbb{R}_+\cdot \varpi$, then
$N_m=J_m\circ \mathrm{val}_m\big(\ker \Phi(m)\big)$ for every $m\in M_\varpi$.
\end{lem}

\begin{proof}
Suppose $m\in M_\varpi$ and $\mathbf{v}\in T_mM_\varpi$. Then $\Phi(m)=\lambda \,\varpi$
for some $\lambda>0$, and $d_m\Phi(\mathbf{v})=b\,\varpi$ for some $b\in \mathbb{R}$; therefore,
$d_m\Phi(\mathbf{v})=a\,\varpi$, $a=b/\lambda$. If $\eta\in \mathfrak{t}$,
$$
g_m\Big(\mathbf{v},J_m\big(\eta_M(m)\big)\Big)=\omega_m\big(\eta_M(m),\mathbf{v}\big)=\langle
d_m\Phi(\mathbf{v}),\eta\rangle =a\,\langle
\Phi(m),\eta\rangle .
$$
Thus $J_m\circ\mathrm{val}_m\big(\ker \Phi(m)\big)\subseteq N_m$; the statement follows by dimension reasons in
view of Corollary \ref{cor:transverse-conic-mm}.
\end{proof}

\subsection{Heisenberg local coordinates}
\label{subsect:heis}
We shall rely on the notion of Heisenberg local coordinates on $X$, for which we refer to \cite{sz}.
If $\gamma$ is a set of Heisenberg local coordinates on $X$ centered at $x$, we shall set
$x+(\theta,\mathbf{v})=:\gamma (\theta,\mathbf{v})$; here $\theta\in (-\pi,\pi)$ and
$\mathbf{v}\in B_{2\mathrm{d}}(\mathbf{0},\delta)$,
the open ball of center the origin and radius $\delta>0$ in $\mathbb{C}^\mathrm{d}\cong \mathbb{R}^{2\mathrm{d}}$.
We shall also write $x+\mathbf{v}$ for $x+(0,\mathbf{v})$.

We may regard $\mathfrak{p}_m(\mathbf{v})=:\pi\big(\gamma(0,\mathbf{v})\big)$
as a set of preferred local coordinates on $M$ centered at
$m=\pi(x)$ \cite{sz}, and the standard circle action $r:S^1\times X\rightarrow X$
is expressed by translation in $\theta$: where defined, we have
$$
r_\beta\big(x+(\theta,\mathbf{v})\big)=x+(\theta+\beta,\mathbf{v}),
$$
where we identify $\beta\in (-\pi,\pi)$ with $e^{i\beta}$.
In particular, these
coordinates come with a unitary isomorphisms $T_mM\cong \mathbb{C}^\mathrm{d}$ (the unitary structure on $\mathbb{C}^\mathrm{d}$
being, of course, the standard one);
furthermore, they are horizontal at $x$
with respect  to the connection 1-form, meaning that the image of the local section $\mathbf{v}\mapsto x+(0,\mathbf{v})$
has horizontal tangent space at $x$. Therefore, a system of Heisenberg local coordinates centered at $X$ determines
an linear isometry $T_xX\cong \mathbb{R}\oplus \mathbb{C}^\mathrm{d}$.

Any $\xi\in \mathfrak{t}$ induces smooth vector fields $\xi_M$ and $\xi_X$ on $M$ and $X$, respectively. If $\xi_M^\sharp$
is the horizontal lift of $\xi_M$, and $\partial/\partial \theta$ is the generator of the standard $S^1$-action, then
\begin{equation}
\label{eqn:inf-generator-lift}
\xi_X=\xi_M^\sharp-\langle \Phi,\xi\rangle \,\frac{\partial}{\partial \theta}.
\end{equation}
In particular, under the linear isometry just mentioned,
$\xi_M(m)\in \mathbb{C}^\mathrm{d}$, and
$\xi_X(x)=\big(-\langle \Phi,\xi\rangle ,\xi_M^\sharp(m)\big)\in \mathbb{R}\oplus\mathbb{C}^\mathrm{d}$.

Suppose $\mathrm{g}=1$.
Let us denote a point in $\mathbb{T}^1$ by its standard angular coordinate $\vartheta$ ($-\pi<\vartheta<\pi$),
and set $\xi=\left.\partial/\partial \vartheta\right|_0$;
then for $\vartheta\sim 0$ we have
\begin{equation}
\label{eqn:T-1-local-action}
\widetilde{\mu}_{-\vartheta}(x)=x+\big(\vartheta\,\Phi(m),-\vartheta\,\xi_M(m)\big)+O\left(\vartheta^2\right).
\end{equation}

More generally, for any $\mathrm{g}\ge 1$
let $\vartheta=(\vartheta_1,\ldots,\vartheta_\mathrm{g})$ be the collective angular coordinate on
$\mathbb{T}^\mathrm{g}$ ($-\pi<\vartheta_j<\pi$), and set
$\xi_j=\left.\partial/\partial \vartheta_j\right|_{\mathbf{0}}$. Also, let
\begin{center}
$\Phi_j=:\langle \Phi,\xi_j\rangle :M\rightarrow \mathbb{R}$,
$\vartheta \cdot \Phi=:\sum_{j=1}^\mathrm{g}\vartheta_j\,\Phi_j$,
$\vartheta\cdot \xi_M=:\sum_{j=1}^\mathrm{g}\vartheta_j\,\xi_{jM}$.
\end{center}
Then for $\vartheta\sim \mathbf{0}$
\begin{equation}
\label{eqn:T-g-local-action}
\widetilde{\mu}_{-\vartheta}(x)=x+\big(\vartheta\cdot\Phi(m),-\vartheta\cdot\xi_M(m)\big)+O\left(\|\vartheta\|^2\right).
\end{equation}
This may be refined as follows.

\begin{lem}
\label{lem:more-refined-local-estimate}
For $\vartheta\sim \mathbf{0}$, we have
$$
\widetilde{\mu}_{-\vartheta}(x)=x+\Big(\vartheta\cdot\Phi(m)+O\left(\|\vartheta\|^3\right),
-\vartheta\cdot\xi_M(m)+O\left(\|\vartheta\|^2\right)\Big).
$$
\end{lem}

\begin{proof}
Given $\vartheta\in \mathbb{R}^\mathrm{g}$ of unit norm, let us define smooth maps
$\gamma :\mathbb{R}\rightarrow M$ and $\widetilde{\gamma}:\mathbb{R}\rightarrow X$ by setting
$\gamma(\tau)=:\mu_{-\tau \vartheta}(m)$ and $\widetilde{\gamma}(\tau)=:\widetilde{\mu}_{-\tau\vartheta}(x)$.
In the induced preferred local coordinates centered at $m$, obviously
$\gamma(\tau)=-\tau\,\vartheta\cdot \xi_M(m)+O\left(\tau^2\right)$ as $\tau\sim 0$.
Also, let $\gamma^\sharp:\mathbb{R}\rightarrow X$ be the unique horizontal lift of $\gamma$ such that
$\gamma^\sharp(0)=x$. In view of Lemma 2.4 of \cite{dp}, for $\tau\sim 0$ we have
\begin{equation}
\label{eqn:horizontal-local}
\gamma^\sharp(\tau)=x+\Big(O\left(\tau^3\right),-\tau\,\vartheta\cdot \xi_M(m)+O\left(\tau^2\right)\Big).
\end{equation}
Since $\Phi$ is constant along $\gamma$, (\ref{eqn:inf-generator-lift}) implies
$\widetilde{\gamma}(\tau)=r_{\tau\vartheta\cdot \Phi(m)}\left(\gamma^\sharp(\tau)\right)$. By
(\ref{eqn:horizontal-local}), we get
$$
\widetilde{\gamma}(\tau)=
x+\Big(\tau\vartheta\cdot \Phi(m)+O\left(\tau^3\right),-\tau\,\vartheta\cdot \xi_M(m)+O\left(\tau^2\right)\Big).
$$
\end{proof}

We shall need a further strengthening of this. Having fixed a system $\gamma=\gamma_x$ of Heisenberg local coordinates
centered at $x$, we can find an open neighborhood $X'\subseteq X$ of $x$ and smoothly varying
family $\gamma_{x'}:(-\pi,\pi)\times B_{2\mathrm{d}}(\mathbf{0},\delta)\rightarrow X$
of Heisenberg local coordinates centered at points $x'\in X'$. We shall write $x'+(\theta,\mathbf{v})=
\gamma_{x'}(\theta,\mathbf{v})$. We may as well suppose that $X'$ is $S^1$-invariant.

\begin{lem}
\label{lem:comparison-heisenberg-local}
For $\mathbf{u},\mathbf{v}\sim \mathbf{0}$ in $ \mathbb{C}^\mathrm{d}$ and $\theta\in (-\pi,\pi)$
we have
$$
x+(\theta,\mathbf{u})=(x+\mathbf{w})+\Big(\theta+\omega_m(\mathbf{w},\mathbf{u})+O\left(\|(\mathbf{u},\mathbf{v})\|^3\right),
\mathbf{u}-\mathbf{w}+O\left(\|(\mathbf{u},\mathbf{v})\|^2\right)\Big).
$$
\end{lem}

Here $\omega_m$ is the symplectic form on $T_mM$, identified with the standard symplectic structure on
$\mathbb{C}^\mathrm{d}$ under the given unitary isomorphism.

\begin{proof}
Since in any Heisenberg local chart the $S^1$ action $r_\theta$ is expressed by a translation by
$\theta$ in the angular coordinates, we may apply $r_{-\theta}$ to both sides and reduce to the
case $\theta=0$.

Each Heisenberg local chart $\gamma_{x'}$ determines a preferred local chart
$\mathfrak{p}_{m'}:B_{2\mathrm{d}}(\mathbf{0},\delta)\rightarrow M$ on $M$ centered at
$m'=:\pi(x')$; we shall write $\mathfrak{p}_{m'}(\mathbf{v})=m'+\mathbf{v}$. Then
$(m+\mathbf{u})+\mathbf{w}=m+\left(\mathbf{u}+\mathbf{w}+O\left(\|(\mathbf{u},\mathbf{w})\|^2\right)\right)$
as $\mathbf{u},\mathbf{w}\sim \mathbf{0}$ in $\mathbb{C}^\mathrm{d}$.

Since $m'+\mathbf{u}=\pi\big(x'+(\theta,\mathbf{u})\big)$ if $m'=\pi(x')$, this implies
\begin{equation}
\label{eqn:separate-coordinates}
x+\mathbf{u}=(x+\mathbf{w})+\big(\beta(\mathbf{u},\mathbf{w}),\mathbf{u}-\mathbf{w}+O\left(\|(\mathbf{u},\mathbf{w}\|^2\right)
\end{equation}
for some smooth function $\beta:B_{2\mathrm{d}}(\mathbf{0},\delta)\times B_{2\mathrm{d}}(\mathbf{0},\delta)\rightarrow
(-\pi,\pi)$ such that $\beta(\mathbf{u},\mathbf{0})=0$. Let us write $\beta=\beta_1+\beta_2+\beta_3$, where
$\beta_1$ is linear and $\beta_2$ is homogenous of degree $2$ in
$(\mathbf{u},\mathbf{v})\in \mathbb{R}^{2\mathrm{d}}\times \mathbb{R}^{2\mathrm{d}}$,
while $\beta_3$ vanishes at the origin to order $\ge 3$. Given $(\mathbf{u},\mathbf{v})$ of unit length,
by Theorem 3.1 of \cite{sz} we have
\begin{eqnarray}
\label{eqn:1st-asympt}
\lefteqn{\Pi_k\left(x+\frac{\mathbf{w}}{\sqrt{k}},x+\frac{\mathbf{u}}{\sqrt{k}}\right)}\\
&=&\left(\frac k\pi\right)^{\mathrm{d}}\,
e^{-i\omega_m(\mathbf{w},\mathbf{u})-\frac 12\,\|\mathbf{u}-\mathbf{w}\|_m^2}\cdot \left(1+O\left(k^{-1/2}\right)\right).
\nonumber
\end{eqnarray}
Adopting $x+\frac{\mathbf{w}}{\sqrt{k}}$ as reference point, we get instead
\begin{eqnarray}
\label{eqn:2nd-asympt}
\lefteqn{\Pi_k\left(x+\frac{\mathbf{w}}{\sqrt{k}},x+\frac{\mathbf{u}}{\sqrt{k}}\right)}\\
&=&\Pi_k\left(x+\frac{\mathbf{w}}{\sqrt{k}},\left(x+\frac{\mathbf{w}}{\sqrt{k}}\right)
+\left(\beta\left(\frac{\mathbf{u}}{\sqrt{k}},\frac{\mathbf{w}}{\sqrt{k}}\right),
\frac{1}{\sqrt{k}}\,(\mathbf{u}-\mathbf{w})+O\left(\frac 1k\right)\right)\right)\nonumber\\
&=&
\left(\frac k\pi\right)^{\mathrm{d}}\,e^{-ik\beta\left(\frac{\mathbf{u}}{\sqrt{k}},\frac{\mathbf{w}}{\sqrt{k}}\right)}\,
e^{-\frac 12\,\|\mathbf{u}-\mathbf{w}\|_m^2}\cdot \left(1+O\left(k^{-1/2}\right)\right)\nonumber\\
&=&\left(\frac k\pi\right)^{\mathrm{d}}\,e^{-i\sqrt{k}\beta_1(\mathbf{u},\mathbf{w})-i\beta_2(\mathbf{u},\mathbf{w})}\,
e^{-\frac 12\,\|\mathbf{u}-\mathbf{w}\|_m^2}\cdot \left(1+O\left(k^{-1/2}\right)\right)\nonumber\\
\nonumber
\end{eqnarray}
Comparing (\ref{eqn:1st-asympt}) and (\ref{eqn:2nd-asympt}) yields
$\beta_1=0$ and $\beta_2=\omega_m$.
\end{proof}

\begin{cor}
\label{cor:yet-more-refined-local-estimate}
As $(\vartheta,\mathbf{v})\sim \mathbf{0}$ in $\mathbb{R}^\mathrm{g}\times \mathbb{C}^\mathrm{d}$, we have
\begin{eqnarray*}
\widetilde{\mu}_{-\vartheta}(x+\mathbf{v})
&=&x+\Big(\vartheta\cdot \Phi(m)
+\omega_m\big(\vartheta\cdot \xi_M(m),\mathbf{v}\big)+O\left(\|(\vartheta,\mathbf{v})\|^3\right),\\
&&\mathbf{v}-\vartheta\cdot \xi_M(m)+O\left(\|(\vartheta,\mathbf{v})\|^2\right)\Big).
\end{eqnarray*}
\end{cor}

\begin{proof}
By Lemma \ref{lem:more-refined-local-estimate},
\begin{eqnarray}
\label{eqn:azione-traslata}
\lefteqn{\widetilde{\mu}_{-\vartheta}(x+\mathbf{v})=}\\
&&(x+\mathbf{v})+\Big(\vartheta\cdot\Phi(m+\mathbf{v})+O\left(\|\vartheta\|^3\right),
-\vartheta\cdot\xi_M(m+\mathbf{v})+O\left(\|\vartheta\|^2\right)\Big).\nonumber
\end{eqnarray}
Now $\vartheta\cdot\xi_M(m+\mathbf{v})=\vartheta\cdot\xi_M(m)+O\left(\|(\vartheta,\mathbf{v})\|^2\right)$.
On the other hand,
\begin{eqnarray}
\label{eqn:derivata-Phi}
\vartheta\cdot\Phi(m+\mathbf{v})&=&\sum_{j=1}^\mathrm{g}\vartheta_j\Phi_j(m+\mathbf{v})\nonumber\\
&=&\sum_{j=1}^\mathrm{g}\vartheta_j\,\Big(\Phi_j(m)+d_m\Phi_j(\mathbf{v})+O\left(\|\mathbf{v}\|^2\right)\Big)\nonumber\\
&=&\sum_{j=1}^\mathrm{g}\vartheta_j\,\Big(\Phi_j(m)+2\,\omega_m\big(\xi_{jM}(m),\mathbf{v})\Big)
+O\left(\|(\vartheta,\mathbf{v}\|^3\right)\nonumber\\
&=&\vartheta\cdot \Phi(m)+2\,\omega_m\big(\vartheta\cdot \xi_M(m),\mathbf{v}\big)+O\left(\|(\vartheta,\mathbf{v}\|^3\right).
\end{eqnarray}
Hence in the Heisenberg chart $\gamma_{x+\mathbf{v}}$ (\ref{eqn:azione-traslata}) may be rewritten
\begin{eqnarray}
\label{eqn:azione-traslata-bis}
\lefteqn{\widetilde{\mu}_{-\vartheta}(x+\mathbf{v})=(x+\mathbf{v})}\\
&&+\Big(\vartheta\cdot \Phi(m)+2\,\omega_m\big(\vartheta\cdot \xi_M(m),\mathbf{v}\big)+O\left(\|(\vartheta,\mathbf{v})\|^3\right),
-\vartheta\cdot\xi_M(m)+O\left(\|(\vartheta,\mathbf{v})\|^2\right).\nonumber
\end{eqnarray}
In view of Lemma \ref{lem:comparison-heisenberg-local}, in the Heisenberg chart $\gamma_x$ (\ref{eqn:azione-traslata-bis}) is
\begin{eqnarray*}
\label{eqn:azione-traslata-tris}
\widetilde{\mu}_{-\vartheta}(x+\mathbf{v})&=&x+\Big(\vartheta\cdot \Phi(m)+2\,\omega_m\big(\vartheta\cdot \xi_M(m),\mathbf{v}\big)
-\omega_m\big(\mathbf{v},\mathbf{v}-\vartheta\cdot \xi_M(m)\big)\nonumber\\
&&+O\left(\|(\vartheta,\mathbf{v}\|^3\right),
\mathbf{v}-\vartheta\cdot\xi_M(m)+O\left(\|(\vartheta,\mathbf{v})\|^2\right)\Big).
\end{eqnarray*}
The statement follows.
\end{proof}

We conclude this section with the following:

\begin{lem}
\label{lem:heis-tg}
Suppose $x\in X$, $m=\pi(x)$, and $t\in T_m$.
If $\upsilon=(\theta,\mathbf{v})\in T_xX$,
then $d_x\widetilde{\mu}_t(\upsilon)=\big(\theta,d_m\mu_t(\mathbf{v})\big)$.
\end{lem}

\begin{proof}
As a smooth path
$\gamma^\sharp :(-\epsilon,\epsilon)\rightarrow X$ tangent to $(0,\mathbf{v})$ at $\tau=0$ we can take the
horizontal lift through $x$ of any path $\gamma:(-\epsilon,\epsilon)\rightarrow M$ tangent to
$\mathbf{v}$ at $\tau=0$. Thus,
$\widetilde{\gamma}(\tau)=:r_{\tau\theta}\left(\gamma^\sharp(\tau)\right)$ is tangent to $\upsilon$ at $\tau=0$,
where $r$ denotes the standard circle action on $X$.
It follows also that $\widetilde{\mu}_t\big(\widetilde{\gamma}(\tau)\big)$ ($\tau\in (-\epsilon,\epsilon)$)
is tangent to $d_x\widetilde{\mu}_t(\upsilon)$
at $\tau=0$. On the other hand, since $\widetilde{\mu}$ and $r$ commute, we have:
\begin{equation}
\label{eqn:hor-lift}
\widetilde{\mu}_t\big(\widetilde{\gamma}(\tau)\big)=\widetilde{\mu}_t\Big(r_{\tau\theta}\big(\gamma^\sharp(\tau)\big)\Big)=
r_{\tau\theta}\Big(\widetilde{\mu}_t\big(\gamma^\sharp(\tau)\big)\Big).
\end{equation}
Since $\widetilde{\mu}$ preserves the connection, $\widetilde{\mu}_t\big(\gamma^\sharp(\tau)\big)$
is the unique horizontal lift through $x$ of the path $\mu_t\big(\gamma(\tau)\big)$, and therefore it is
tangent to $\big(0,d_m\mu_t(\mathbf{v})\big)$ at $\tau=0$. The statement follows from this and (\ref{eqn:hor-lift}).
\end{proof}

\subsection{Some linear algebra}

The proofs of the following statements are left to the reader.

\begin{lem}
\label{lem:key-det-1}
Given $n\ge 1$, suppose $\mathbf{v}\in \mathbb{R}^n$ and let $C$ be a symmetric non-singular
$n\times n$ matrix. Then the symmetric $(n+1)\times (n+1)$ matrix
$$
D=:\left(
     \begin{array}{cc}
       0 & \mathbf{v}^t \\
       \mathbf{v} & C \\
     \end{array}
   \right)
$$
has determinant $\det(D)=-\left(\mathbf{v}^t C^{-1}\mathbf{v}\right)\cdot\det(C)$.
\end{lem}

A variant of Lemma \ref{lem:key-det-1} is as follows:

\begin{lem}
\label{lem:key-det-3}
Let $(V,\Psi)$ be an $n$-dimensional Euclidean vector space, let $(V^\vee,\Psi^\vee)$ be its dual space, endowed with
the induced Euclidean structure. Suppose
$\phi\in V^\vee\setminus\{ 0\}$.
Let $\Omega:V\times V\rightarrow \mathbb{R}$ be a scalar product, and suppose that
the restriction $\Omega_\phi:\ker(\phi)\times \ker(\phi)\rightarrow \mathbb{R}$ of $\Omega$ to $\ker(\phi)$ is non-degenerate.
Let $\mathcal{B}=(v_1,\ldots,v_n)$ be any orthonormal basis of $(V,\Psi)$, and let $\mathcal{B}^*=(v_1^*,\ldots,v_n^*)$
be the dual basis. Let $C$ be the matrix of
$\Omega$ with respect to $\mathcal{B}$.
If $\phi=\sum_j\phi_j\,v_j^*$,
let us write $\phi_\mathcal{B}=(\phi_1,\ldots,\phi_n)^t$. Then
$$
\left|
  \begin{array}{cc}
    0 & \phi_\mathcal{B}^t \\
    \phi_\mathcal{B} & C \\
  \end{array}
\right|=-\|\phi\|^2_\Psi\,\det{}\!_\Psi\big(\Omega_\phi),
$$
where $\det _\Psi\big(\Omega_\phi)$ is the determinant of the matrix of $\Omega_\phi$
with respect to any orthonormal basis of $\ker(\phi)$ (orthonormal with respect to $\Psi$).
\end{lem}

Here, $\|\phi\|^2_\Psi=\sum_j\phi_j^2$ is the squared norm of $\phi$ with respect to $\Psi^\vee$.

Let us now suppose that $\Omega$ is also positive definite.
Comparing Lemmata \ref{lem:key-det-1} and \ref{lem:key-det-3}, we get

\begin{lem}
\label{lem:key-det-4}
Let $V$ be an $n$-dimensional vector space, and let $\Psi$, $\Omega$ be two Euclidean structures on
$V$. Given any $\phi\in V^\vee\setminus \{0\}$, we have
$$
\|\phi\|^2_\Omega\cdot \det{}\!_\Psi\big(\Omega)=\|\phi\|^2_\Psi\,\det{}\!_\Psi\big(\Omega_\phi),
$$
where $\det{}\!_\Psi\big(\Omega)$ is the determinant of the matrix representing $\Omega$ with respect to
any orthonormal basis of $(V,\Psi)$.
\end{lem}

\section{Proof of Theorem \ref{thm:circle-case}}
\begin{proof}
Let us prove the first statement. If $\Phi>0$, for any $k\le 0$ and $\ell \ge 0$ we have
$k\not\in \ell \Phi(M)$. Hence, $\widetilde{H}_k(X)\cap H_\ell(X)=\{0\}$ \cite{gs1}.
It follows that
$$
\widetilde{H}_k(X)=\bigoplus _{\ell \ge 0}\widetilde{H}_k(X)\cap H_\ell(X)=\{0\}.
$$

Let us now prove the third statement.
We shall adapt the arguments in \cite{z}, \cite{bsz} and \cite{sz}
for the standard circle action.
In Heisenberg local coordinates centered at $x$, let us set
$$x_{jk}=:x+\left(\frac{\theta_j}{\sqrt{k}},\frac{\mathbf{v}_j}{\sqrt{k}}\right)$$
for $j=1,2$.
For any $x\in X$, we have
\begin{eqnarray}
\label{eqn:equiv-proj-circle-case}
\widetilde{\Pi}_k\left(x_{1k},
x_{2k}\right)=\frac{1}{2\pi}\,\int_{-\pi}^{\pi}e^{-ik\vartheta}\,
\Pi\Big(\widetilde{\mu}_{-\vartheta}\left(x_{1k}\right),
x_{2k}\Big)\,\mathrm{d}\vartheta,
\end{eqnarray}
where we work in the standard angular coordinate and write $\vartheta$ for $e^{i\vartheta}$.

For $m=\pi(x)$, set $r_m=:|T_m|$, and suppose $T_m=\left\{e^{ib_1},\ldots,e^{ib_{r_m}}\right\}$
with $-\pi<b_e<\pi$ (if $-1\in T_m$, we need only shift integration to $(-\pi+\delta,\pi+\delta)$
for some suitably small $\delta>0$).
For some sufficiently small $\epsilon>0$, let $\varrho>0$ be a bump function,
supported in
$\bigcup_{e=1}^{r_m}(b_e-\epsilon,b_e+\epsilon)$ and identically equal to $1$ on
$\bigcup_{e=1}^{r_m}(b_e-\epsilon/2,b_e+\epsilon/2)$. Then
$$
\widetilde{\Pi}_k(x_{1k},x_{2k})=
\widetilde{\Pi}_k(x_{1k},x_{2k})^{(1)}+\widetilde{\Pi}_k(x_{1k},x_{2k})^{(2)},
$$ where
$\widetilde{\Pi}_k(x_{1k},x_{2k})^{(1)}$ (respectively, $\widetilde{\Pi}_k(x_{1k},x_{2k})^{(2)}$) is defined as in
(\ref{eqn:equiv-proj-circle-case}), with the integrand multiplied by $\varrho$
(respectively, by $1-\varrho$).

\begin{lem}
\label{lem:1st-reduction}
$\widetilde{\Pi}_k(x_{1k},x_{2k})^{(2)}=O\left(k^{-\infty}\right)$ as $k\rightarrow \infty$.
\end{lem}

\begin{proof}
There exist $\delta_1,\delta_2>0$ such that $\mathrm{dist}_X\big(\widetilde{\mu}_{-\vartheta}\left(y\right),
y'\big)\ge \delta_1$ whenever $\mathrm{dist}_X(x,y), \,\mathrm{dist}_X(x,y')\le \delta_2$ and
$\vartheta\in \mathrm{supp}(1-\varrho)$.
Since the singular support of $\Pi$ is the diagonal in $X\times X$ (\cite{f}, \cite{bg}),
in the same range
$$
f_{yy'}(\vartheta)=:\big(1-\varrho(\vartheta)\big)\,\Pi\left(\widetilde{\mu}_{-\vartheta}(x_{1k}),x_{2k}\right)
$$
is $\mathcal{C}^\infty$ on $S^1$, and its $k$-th Fourier coefficient $\widetilde{\Pi}_k(y,y')^{(2)}$
is uniformly $O\left(k^{-\infty}\right)$ as $k\rightarrow \infty$. Then we need only set $y=x_{1k}$,
$y'=x_{2k}$.
\end{proof}

Let $\sim$ stand for \lq has the same asymptotics as\rq; Lemma \ref{lem:1st-reduction} implies
$\widetilde{\Pi}_k(x_{1k},x_{2k})\sim\widetilde{\Pi}_k(x_{1k},x_{2k})^{(1)}$ as $k\rightarrow +\infty$.
On the support of $\varrho$, $\widetilde{\mu}_{-\vartheta}(x_{1k})$ belongs to a small neighborhood
of $x$, and there we may represent $\Pi$ as an FIO of the form
\begin{equation}
\label{eqn:szego-as-FIO}
\Pi\left(x',x''\right)=\int_0^{+\infty}e^{it\psi\left(x',x''\right)}\,s\left(t,x',x''\right)\,dt+
S\left(x',x''\right),
\end{equation}
where the phase $\psi$ satisfies $\Im(\psi)\ge 0$ and is essentially determined by the metric,
$s$ is a classical symbol of degree d, and $S$ is $\mathcal{C}^\infty$ \cite{bs}. The argument used for
Lemma \ref{lem:1st-reduction} implies that $S$ contributes negligibly to the asymptotics.

With the change of variable  $t\rightarrow kt$, we conclude that
\begin{eqnarray}
\label{eqn:fio-pi-rescaled-1-dim}
\lefteqn{\widetilde{\Pi}_k(x_{1k},x_{2k})}\\
&\sim&\frac{1}{2\pi}\,\int_0^{+\infty}\int_{-\pi}^\pi\,e^{-ik\vartheta}
e^{it\psi\left(\widetilde{\mu}_{-\vartheta}(x_{1k}),x_{2k}\right)}\,
s\left(t,\widetilde{\mu}_{-\vartheta}(x_{1k}),x_{2k}\right)\,\varrho(\vartheta)
\,\mathrm{d}t\,\mathrm{d}\vartheta\nonumber\\
&=&\frac{k}{2\pi}\,\int_0^{+\infty}\int_{-\pi}^\pi\,
e^{ik\Psi_k(x,t,\vartheta)}\,s\left(kt,\widetilde{\mu}_{-\vartheta}(x_{1k}),x_{2k}\right)\,\varrho(\vartheta)
\,\mathrm{d}t\,\mathrm{d}\vartheta,\nonumber
\end{eqnarray}
where
\begin{equation}
\label{eqn:phase-dim-1}
\Psi_k(x,t,\vartheta)=t\psi\left(\widetilde{\mu}_{-\vartheta}(x_{1k}),x_{2k}\right)-\vartheta.
\end{equation}
Let us denote by $P_k(x)$ the right hand side of (\ref{eqn:fio-pi-rescaled-1-dim}).
For some $C\gg 0$ we choose $\gamma\in \mathcal{C}^\infty_0\big(\big(1/(2C),C\big)\big)$ such that
$\gamma=1$ on $(1/C,C)$, and write $P_k(x)=P_k(x)'+P_k(x)''$, where $P_k(x)'$ (respectively,
$P_k(x)''$) is defined as $P_k(x)$ with the integrand multiplied by $\gamma(t)$
(respectively, by $1-\gamma(t)$).

\begin{lem}
\label{lem:second-reduction-dim-1}
$P_k(x)''=O\left(k^{-\infty}\right)$ as $k\rightarrow +\infty$.
\end{lem}

\begin{proof}
Working near some $b_e$ at a time, let us write $\vartheta=b_e+\theta$, $\theta\sim 0$.
In view of (\ref{eqn:T-1-local-action}), in Heisenberg local coordinates centered at $x$ with $m=\pi(x)$
we have
\begin{equation}
\label{eqn:T-1-local-action-proof}
\widetilde{\mu}_{-\vartheta}(x)=\widetilde{\mu}_{-\theta}(x)=x+\big(\theta\,\Phi(m),-\theta\,\xi_M(m)\big)+O\left(\theta^2\right).
\end{equation}
Now $d_{(y,y)}\psi=\big(\alpha_y,-\alpha_y\big)$ for any $y\in Y$. Therefore, in view of (\ref{eqn:T-1-local-action-proof})
and the definition of Heisenberg local coordinates,
\begin{equation}
\label{eqn:derivata-vartheta-dim-1}
\left.\partial_\vartheta\psi\left(\widetilde{\mu}_{-\vartheta}(x),x\right)\right|_{b_e}=
\left.\partial_\theta\psi\left(\widetilde{\mu}_{-\theta}(x),x\right)\right|_{0}=\Phi(m).
\end{equation}
Consequently, if $0<\epsilon\ll 1$ and $k\gg 0$,
at $\vartheta=b_e+\theta\in \mathrm{supp}(\varrho)$
we have
\begin{equation}
\label{eqn:bound-on-partial-vartheta}
2\,A\ge 2\,\Phi(m)\ge \big|\partial_\theta\psi\left(\widetilde{\mu}_{-\theta}(x_{1k}),x_{2k}\right)\big|
\ge \frac 12\,\Phi(m)\ge \frac a2,
\end{equation}
where $a=\min\|\Phi\|$, $A=\max\|\Phi\|$.

Thus for $t\le 1/C$ we have
\begin{equation*}
\big|\partial_\vartheta\Psi_k(x,t,\vartheta)\big|=\big|t\,\partial_\vartheta
\psi\left(\widetilde{\mu}_{-\vartheta}(x_{1k}),x_{2k}\right)-1\big|
\ge 1-2\,\frac AC\ge \frac 12,
\end{equation*}
while for $t\ge C$
$$
\big|\partial_\vartheta\Psi_k(x,t,\vartheta)\big|=\big|t\,\partial_\vartheta\psi\left(\widetilde{\mu}_{-\vartheta}(x_{1k}),x_{2k}\right)-1\big|
\ge t\,\frac a2-1\ge \frac 14\,at+\left(\frac 14 \,aC-1\right).
$$
The statement follows integrating by parts in $d\vartheta$.
\end{proof}

We can write $\varrho=\sum_{e=1}^{r_m}\varrho_e$, where $\varrho_e\in \mathcal{C}^\infty_0\big((b_e-\epsilon,b_e+\epsilon)\big)$.
In view of Lemma \ref{lem:second-reduction-dim-1},
\begin{eqnarray}
\label{eqn:variuospieces}
\lefteqn{\widetilde{\Pi}_k(x_{1k},x_{2k})}\\
&\sim&\sum_{e=1}^{r_m}\frac{k}{2\pi}\int_{1/2C}^{2C}\int_{b_e-\epsilon}^{b_e+\epsilon}
e^{ik\Psi_k(x,t,\vartheta)}\,s\left(k t,\widetilde{\mu}_{-\vartheta}(x_{1k}),x_{2k}\right)\,\varrho_e(\vartheta)\,\gamma(t)\,\mathrm{d}t\,\mathrm{d}\vartheta.
\nonumber
\end{eqnarray}
Let us study the asymptotics of the $e$-th summand $S_e(k)$ in (\ref{eqn:variuospieces}).
On the support of $\varrho_e$, we keep writing $\vartheta=b_e+\eta$, with
$|\eta|<\epsilon$. Then
\begin{eqnarray}
\label{eqn:S-e-}
\lefteqn{S_e(k)=e^{-ikb_e}}\\
&&\cdot\frac{k}{2\pi}\int_{1/2C}^{2C}\int_{-\epsilon}^{\epsilon}
e^{ik\Psi^{(e)}_k(x,t,\eta)}\,s\big(k t,\widetilde{\mu}_{-\eta-b_e}(x_{1k}),x_{2k}\big)\,\varrho_0(\eta)\,\gamma(t)\,\mathrm{d}t\,\mathrm{d}\eta,
\nonumber
\end{eqnarray}
where now
\begin{equation}
\label{eqn:phase-dim-1-e}
\Psi^{(e)}_k(x,t,\eta)=t\psi\big(\widetilde{\mu}_{-\eta-b_e}(x_{1k}),x_{2k}\big)-\eta.
\end{equation}
We have assumed $\varrho_e(\vartheta)=\varrho_0(\vartheta-b_e)$.

Let $f\in \mathcal{C}_0^\infty(\mathbb{R})$ be a bump function with $f(\eta)=1$ for $|\eta|<1/2$ and $f(\eta)=0$ for
$|\eta|>1$, and set $f_k(\eta)=:f\left(k^{7/18}\eta\right)$. Inserting the identity $1=f_k(\eta)+\big(1-f_k(\eta)\big)$ in
(\ref{eqn:S-e-}), we get $S_e(k)=S_e'(k)+S_e''(k)$, where $S_e'(k)$ (resp., $S_e'(k)''$)
is defined as in (\ref{eqn:S-e-}), except that the integrand has been
multiplied by $f_k(\eta)$ (resp., $1-f_k(\eta)$).

\begin{lem}
\label{lem:S-k''}
$S_e''(k)=O\left(k^{-\infty}\right)$ as $k\rightarrow +\infty$.
\end{lem}

\begin{proof}
Let us set $\upsilon_j=:(\theta_j,\mathbf{v}_j)\in \xi_X(x)^\perp$.
Choosing as preferred local coordinates on $M$ centered at $m=\pi(x)$
those given by the exponential map, we have
$$
\mu_{-\mathbf{b}_e}(m+\mathbf{v})=m+\mathbf{v}^{(e)},
$$
where $\mathbf{v}^{(e)}=:d_m\mu_{-\mathbf{b}_e}(\mathbf{v})$.
By Lemma 2.8 of \cite{p1},
\begin{equation}
\label{eqn:tilde-mu-e}
\widetilde{\mu}_{-\mathbf{b}_e}\big(x+(\theta,\mathbf{v})\big)=
x+\left(\theta+O\left(\|\mathbf{v}\|^3\right),\mathbf{v}^{(e)}\right).
\end{equation}

Let us set $\upsilon_j=:(\theta_1,\mathbf{v}_j)\in \xi_X(x)^\perp$,
$\upsilon_j^{(e)}=:\left(\theta_1,\mathbf{v}_j^{(e)}\right)\in \xi_X(x)^\perp$. To first order in
$(\theta,\upsilon_j)$, we have
\begin{equation}
\label{eqn:tilde-mu-e-1}
\widetilde{\mu}_{-\vartheta}\big(x+\upsilon_1\big)\sim
x+\Big(-\eta\,\xi_X(x)+\upsilon_1^{(e)}\Big).
\end{equation}

Since $x_{jk}=x+\upsilon_j/\sqrt{k}$, and because
Heisenberg local coordinates are isometric at the origin,
$$
\mathrm{dist}_X\big(\widetilde{\mu}_{-\vartheta}(x_{1k}),x_{2k}\big)\ge \frac 12\,\sqrt{\eta^2\,\|\xi_X(x)\|^2+
\frac 1k\,\|\upsilon_1^{(e)}-\upsilon_2\|^2}\ge C\,|\eta|,
$$
for some $C>0$. In particular, where $1-f_k(\eta)\neq 0$ the latter distance is $\ge (C/2)\,k^{-7/18}$.
Now we need only apply integration by parts by $t$, as in the proof of Lemma 2.5 of \cite{p2}.
\end{proof}

Thus $S_e(k)\sim S_e(k)''$ as $k\rightarrow +\infty$. To evaluate the latter, let us perform the change of
variable of integration $\eta\rightsquigarrow \eta/\sqrt{k}$. Then, perhaps after changing the definition of $C$,
\begin{eqnarray}
\label{eqn:S-e-rescaled}
\lefteqn{S_e(k)\sim S_e''(k)=\frac{1}{2\pi}\,e^{-ikb_e}\,k^{1/2}}\\
&&\cdot \int_{1/2C}^{2C}\int_{-C\,k^{1/9}}^{Ck^{1/9}}e^{ik\Psi^{(e)}_k\big(x,t,\eta/\sqrt{k}\big)}\,s\big(k t,\widetilde{\mu}_{-\eta/\sqrt{k}-b_e}(x_{1k}),x_{2k}\big)\, f\left(\frac{\eta}{k^{1/9}}\right)\,
\gamma(t)\,\mathrm{d}t\,\mathrm{d}\eta.
\nonumber
\end{eqnarray}

In view of Corollary \ref{cor:yet-more-refined-local-estimate}, we have
\begin{eqnarray}
\label{eqn:action-expanded-1}
\lefteqn{\widetilde{\mu}_{-\eta/\sqrt{k}-b_e}(x_{1k})}\\
&=&x+\left( \frac{1}{\sqrt{k}}\,\big(\eta\,\Phi(m)+\theta_1\big)+\frac{1}{k}\,\omega_m\big(\eta\, \xi_M(m),\mathbf{v}_1^{(e)}\big)
+B_3\left(\frac{\eta}{\sqrt{k}},\frac{\mathbf{v}}{\sqrt{k}}\right),\right.
\nonumber\\
&&\left.+\frac{1}{\sqrt{k}}\,\left(\mathbf{v}_1^{(e)}-\eta\, \xi_M(m)\right)+B_2\left(\frac{\eta}{\sqrt{k}},
\frac{\mathbf{v}}{\sqrt{k}}\right)\right),\nonumber
\end{eqnarray}
where here and in the following $B_j$ denotes a smooth function, with appropriate codomain, vanishing to
$j$-th order at the origin in $\mathbb{R}\times \mathbb{C}^\mathrm{d}$.

Let us define
\begin{eqnarray}
\label{eqn:funzione-A}
\lefteqn{A_{ek}(\eta,\upsilon_1,\upsilon_2)}\\
&=:&\frac{1}{\sqrt{k}}\,\big(\eta\,\Phi(m)+\theta_1-\theta_2\big)+\frac{1}{k}\,\omega_m\left(\eta\, \xi_M(m),\mathbf{v}_1^{(e)}\right)
+B_3\left(\frac{\eta}{\sqrt{k}},\frac{\mathbf{v}}{\sqrt{k}}\right),\nonumber
\end{eqnarray}
where we set collectively $\mathbf{v}=(\mathbf{v}_1,\mathbf{v}_2)$.

By the discussion in \S 3 of \cite{sz} we deduce from (\ref{eqn:phase-dim-1-e})
and (\ref{eqn:action-expanded-1}):
\begin{eqnarray}
\label{eqn:psi-e-expanded}
\lefteqn{\Psi^{(e)}_k\big(x,t,\eta/\sqrt{k}\big)
=t\psi\big(\widetilde{\mu}_{-\eta/\sqrt{k}-b_e}(x_{1k}),x_{2k}\big)-\eta/\sqrt{k}}\\
&=&it\,\left[1-e^{i\,A_{ek}(\eta,\upsilon_1,\upsilon_2)}\right]
-\frac{it}{k}\psi_2\left(\mathbf{v}_1^{(e)}-\eta\, \xi_M(m),\mathbf{v}_2\right)\,e^{i\,A_{ek}(\eta,\upsilon_1,\upsilon_2)}
-\eta/\sqrt{k}\nonumber\\
&&+
it\,R_3^\psi\left(\frac{1}{\sqrt{k}}\,\left(\mathbf{v}_1^{(e)}-\eta\, \xi_M(m)\right),
\frac{1}{\sqrt{k}}\,\mathbf{v}_2\right)\,e^{i\,A_{ek}(\eta,\upsilon_1,\upsilon_2)},\nonumber
\end{eqnarray}
where $R_3^\psi$ vanishes to third order at the origin. Here,
$$
\psi_2(\mathbf{r},\mathbf{s})=:-i\,\omega_m(\mathbf{r},\mathbf{s})-\frac 12\,\|\mathbf{r}-\mathbf{s}\|^2
\,\,\,\,\,\,\,\,\,\,\left(\mathbf{r},\mathbf{s}\in \mathbb{C}^\mathrm{d}\right).
$$
Now
\begin{eqnarray*}
it\,\left[1-e^{i\,A_{ek}(\eta,\upsilon_1,\upsilon_2)}\right]
&=& \frac{t}{\sqrt{k}}\,\big(\eta\,\Phi(m)+\theta_1-\theta_2\big)+\frac tk\omega_m\left(\eta\,\xi_M(m),\mathbf{v}_1^{(e)}\right) \nonumber\\
&&+\frac{it}{2k}\,\big(\eta\,\Phi(m)+\theta_1-\theta_2\big)^2+B_3'\left(\frac{\eta}{\sqrt{k}},\frac{\mathbf{v}}{\sqrt{k}}\right)
\end{eqnarray*}

Therefore,
\begin{eqnarray}
\label{eqn:phase-1-rescaled-exp}
\lefteqn{ik\,\Psi^{(e)}_k\big(x,t,\eta/\sqrt{k}\big)=i\sqrt{k}\,\Big[t\,\big(\eta\,\Phi(m)+\theta_1-\theta_2\big)-\eta\Big]}\\
&&-\frac{t}{2}\,\big(\eta\,\Phi(m)+\theta_1-\theta_2\big)^2+it\,\omega_m\left(\eta\,\xi_M(m),\mathbf{v}_1^{(e)}\right)+t\,\psi_2\left(\mathbf{v}_1^{(e)}-\eta\, \xi_M(m),\mathbf{v}_2\right)\,e^{i\,A_{ek}(\eta,\upsilon_1,\upsilon_2)}\nonumber\\
&&+t\,k\,B_3''\left(\frac{\eta}{\sqrt{k}},\frac{\mathbf{v}}{\sqrt{k}}\right).\nonumber
\end{eqnarray}

Inserting (\ref{eqn:phase-1-rescaled-exp}) in (\ref{eqn:S-e-rescaled}), we obtain

\begin{eqnarray}
\label{eqn:S-e-rescaled-1}
\lefteqn{S_e(k)\sim \frac{1}{2\pi}\,e^{-ikb_e}\,k^{1/2}}\\
&&\cdot \int_{1/2C}^{2C}\int_{-C\,k^{1/9}}^{Ck^{1/9}}e^{i\sqrt{k}\,\Upsilon(t,\eta,\theta)}\,
e^{G(m,\theta,\mathbf{v},\upsilon)}\,S_k(m,\theta,\mathbf{v},\upsilon)\,\mathrm{d}t\,\mathrm{d}\eta,
\nonumber
\end{eqnarray}
where
\begin{equation}
\label{eqn:upsilon}
\Upsilon(t,\eta,\theta)=:t\,\big(\eta\,\Phi(m)+\theta_1-\theta_2\big)-\eta,
\end{equation}
\begin{eqnarray*}
\lefteqn{G(m,\eta,\mathbf{v},\upsilon)=:-t\,\big(\eta\,\Phi(m)+\theta_1-\theta_2\big)^2}\\
&&+it\,\omega_m\left(\eta\,\xi_M(m),\mathbf{v}_1^{(e)}\right)+t\,\psi_2\left(\mathbf{v}_1^{(e)}-\eta\, \xi_M(m),\mathbf{v}_2\right)\,e^{i\,A_{ek}(\eta,\upsilon_1,\upsilon_2)},
\end{eqnarray*}
and the remaining terms have been incorporated in the phase $S_k$;
notice that $\Re(G)\le -c\,\eta^2$ for some $c>0$.
The amplitude may be Taylor expanded in $\big(\eta/\sqrt{k},\upsilon/\sqrt{k}\big)$, and the error
at the $N$-th step is bounded by $C\,e^{-c\eta^2}\,k^{-aN}$ for some $a>0$. Thus we may integrate the
expansion term by term.

Now we interpret (\ref{eqn:S-e-rescaled-1}) as an oscillatory integral in $\sqrt{k}$, with real phase
$\Upsilon$. Since $\partial_t\Upsilon=\eta\,\Phi(m)+\theta_1-\theta_2$, an integration by parts in
$\mathrm{d}t$ shows that we only lose a rapidly decaying contribution if we multiply the integrand
in (\ref{eqn:S-e-rescaled-1}) by a compactly supported bump function identically $1$ on a neighborhood
of $(\theta_2-\theta_1)/\Phi(m)$.

We see from (\ref{eqn:upsilon}) that $\Upsilon$ has a unique critical point
in
$$
(t_0,\eta_0)=\big(1/\Phi(m),(\theta_2-\theta_1)/\Phi(m)\big),
$$ and that the Hessian
matrix there is
$$
H(\Upsilon)=\left(
              \begin{array}{cc}
                0 & \Phi(m) \\
                \Phi(m) & 0 \\
              \end{array}
            \right);
$$
the determinant being $-\Phi(m)^2<0$, the critical point is non-degenerate.
Hence
$$
\sqrt{\det \left(\frac{\sqrt{k}}{2\pi i}\cdot H(\Upsilon)\right)}=\frac{\sqrt{k}}{2\pi}\,\Phi(m).
$$

Applying the stationary phase Lemma, we obtain for (\ref{eqn:S-e-rescaled-1}) an asymptotic
expansion in descending powers of $k^{1/2}$, with leading order term given by
$$
e^{-ikb_e-i\sqrt{k}\, (\theta_2-\theta_1)/\Phi(m)}\,\Phi(m)^{-(\mathrm{d}+1)}\,\left(\frac k\pi\right)^{\mathrm{d}}\,e^{E(\upsilon_1^{(e)},\upsilon_2)},
$$
where (see (\ref{eqn:phase-1-rescaled-exp}))
\begin{eqnarray*}
\lefteqn{E(\upsilon_1^{(e)},\upsilon_2)=:}\\
&&\frac{1}{\Phi(m)}\left[i\,\omega_m\left(\frac{\theta_2-\theta_1}{\Phi(m)}\,\xi_M(m),\mathbf{v}_1^{(e)}\right)+
\psi_2\left(\mathbf{v}_1^{(e)}-\frac{\theta_2-\theta_1}{\Phi(m)}\, \xi_M(m),\mathbf{v}_2\right)\right]\\
&=&\frac{1}{\Phi(m)}\left\{i\,\left[\omega_m\left(\frac{\theta_2-\theta_1}{\Phi(m)}\,\xi_M(m),\mathbf{v}_1^{(e)}+\mathbf{v}_2\right)
-
\omega_m\left(\mathbf{v}_1^{(e)},\mathbf{v}_2\right)\right]\right.\\
&&\left.-\frac 12\,\left\|\left(\mathbf{v}_1^{(e)}-\mathbf{v}_2\right)
-\frac{\theta_2-\theta_1}{\Phi(m)}\, \xi_M(m)\right\|^2\right\}.
\end{eqnarray*}

Finally, to prove the second statement suppose $x_{1k},\,x_{2k}\in X$ in (\ref{eqn:fio-pi-rescaled-1-dim})
satisfy $\mathrm{dist}_X\big(\widetilde{\mu}_t(x_{1k}),x_{2k}\big)\ge C\,k^{\epsilon-1/2}$ $\forall\,t\in \mathbb{T}^1$.
Then, arguing similarly to the proof of Lemma 2.5 of \cite{p2}, one checks following \cite{bs} that
the phase in (\ref{eqn:phase-dim-1}) satisfies
$\partial_t\Psi_k\ge C'k^{2\epsilon-1}$, and the statement is proved integrating by parts in $dt$.

\end{proof}

\begin{rem}
Since the critical point of the phase $\Upsilon$ in (\ref{eqn:S-e-rescaled-1})
is non-degenerate, the asymptotic expansion derived in
Theorem \ref{thm:circle-case} may be locally smoothly
deformed with $x$ \cite{ms}. In general, however, as we move from a given reference point
$x$ to a nearby $x'$,
some of the critical points may cease to be real, since $T_m$ is only \lq upper semi-continuous\rq.
Thus the contribution to the asymptotics coming from, say, the $e$-th summand
in (\ref{eqn:variuospieces})
might be negligible at most $x'$ near $x$ if $e^{i\mathbf{b}_e}\not\in T_{x'}$ generically.
\end{rem}

\subsection{Proof of Corollary \ref{cor:dim-circle-case}.}
\begin{proof}
For $k=1,2,\ldots$, let us define
$f_k\in \mathcal{C}^\infty(M)$ by setting
$f_k(m)=:(\ell\,k)^{-\mathrm{d}}\,\widetilde{\Pi}_{\ell k}(x,x)$ if $m=\pi(x)$.
Given Theorem \ref{thm:circle-case} (or Corollary \ref{cor:point-wise-circle}),
$f_k\le C\,k^\mathrm{d}$ for some $C>0$ and
$f_k\rightarrow \left(\ell/\pi^\mathrm{d}\right)\,\Phi^{-(\mathrm{d}+1)}$ for $k\rightarrow +\infty$
pointwise on a dense open subset $M'\subseteq M$. By the dominated convergence Theorem,
$$
(\ell\,k)^{-\mathrm{d}}\,\dim\left(\widetilde{H}_{\ell k}(X)\right)=
\int_Mf_k\,dV_M\rightarrow \left(\ell/\pi^\mathrm{d}\right)\,\int_M\Phi^{-(\mathrm{d}+1)}
\,dV_M.
$$
On the other hand, $dV_M/\pi^{\mathrm{d}}=\big(1/\mathrm{d}!\big)\,\big(\omega/\pi\big)^{\mathrm{d}}
=\big(1/\mathrm{d}!\big)\cdot c_1(A)^\mathrm{d}$.
\end{proof}

\section{Proof of Theorem \ref{thm:higher-dim-case}}

\begin{proof}
We start from the relation
\begin{equation}
\label{eqn:equiv-proj-general-case}
\widetilde{\Pi}_{k\varpi}(x,x)=
\frac{1}{(2\pi)^{\mathrm{g}}}\,\int_{-\pi}^\pi\cdots\int_{-\pi}^\pi
e^{-ik\vartheta\cdot \varpi}\,\widetilde{\Pi}\left(\widetilde{\mu}_{-\vartheta}(x),x\right)\,d\vartheta,
\end{equation}
where $\vartheta=\big(\vartheta_1,\ldots,\vartheta_\mathrm{g}\big)\in (-\pi,\pi)^{\mathrm{g}}$,
$\vartheta\cdot \varpi=\sum_{j=1}^\mathrm{g}\vartheta_j\varpi_j$,
$\mathrm{d}\vartheta=\mathrm{d}\vartheta_1\,\cdots\,\mathrm{d}\vartheta_\mathrm{g}$.

For $\mathbf{a}=(a_1,\ldots,a_\mathrm{g})\in \mathbb{R}^\mathrm{g}$, we set $e^{i\mathbf{a}}=:\left(e^{ia_1},\ldots,e^{ia_\mathrm{g}}\right)$.
Suppose $T_m=\left\{e^{i\mathbf{b}_1},\ldots,e^{i\mathbf{b}_{r_m}}\right\}$ with
$\mathbf{b}_e=(b_{e1},\ldots,b_{e\mathrm{g}})\in (-\pi,\pi)^{\mathrm{g}}$ for
$1\le e\le r_m$. Choose $\epsilon>0$ sufficiently small, and let $\varrho\in \mathcal{C}^\infty(\mathbb{T})$
be a bump function supported where $\|\vartheta-\mathbf{b}_e\|\le \epsilon$ for some $e$, and equal to $1$
where  $\|\vartheta-\mathbf{b}_e\|\le \epsilon/2$ for some $e$. Thus $\varrho=\sum_{j=1}^\mathrm{g}\varrho_e$,
where the supports of the $\varrho_e$'s are at positive distance from each other, and each $\varrho_e$
is supported where  $\|\vartheta-\mathbf{b}_e\|\le \epsilon$. The argument in the proof of Lemma \ref{lem:1st-reduction}
implies
\begin{equation}
\label{eqn:equiv-proj-general-case-bump}
\widetilde{\Pi}_{k\varpi}(x,x)\sim
\frac{1}{(2\pi)^{\mathrm{g}}}\,\int_{-\pi}^\pi\cdots\int_{-\pi}^\pi
e^{-ik\vartheta\cdot \varpi}\,\varrho(\vartheta)\,
\widetilde{\Pi}\left(\widetilde{\mu}_{-\vartheta}(x),x\right)\,d\vartheta.
\end{equation}
Again, $\widetilde{\mu}_{-\vartheta}(x)$ is close to $x$ on the support of $\varrho$.
Therefore,  in (\ref{eqn:equiv-proj-general-case-bump}) we may replace $\Pi$ by its
representation as an FIO in (\ref{eqn:szego-as-FIO}), and the contribution of $S$ to the
asymptotics is $O\left(k^{-\infty}\right)$. Rescaling in $t$ as in (\ref{eqn:fio-pi-rescaled-1-dim})
and adapting the proof of Lemma \ref{lem:second-reduction-dim-1}, we obtain the analogue of
(\ref{eqn:variuospieces}):
\begin{eqnarray}
\label{eqn:variuospieces-general}
\lefteqn{\widetilde{\Pi}_{k\varpi}(x,x)}\\
&\sim&\sum_{e=1}^{r_m}\frac{k}{(2\pi)^{\mathrm{g}}}\cdot
\int_{1/2C}^{2C}\int_{-\pi}^\pi\cdots\int_{-\pi}^\pi e^{ik\Psi(x,t,\vartheta)}\,
s\left(k t,\widetilde{\mu}_{-\vartheta}(x),x\right)\,\varrho_e(\vartheta)\,dt\,d\vartheta,
\nonumber
\end{eqnarray}
where now
\begin{equation}
\label{eqn:defn-phase-general}
\Psi(x,t,\vartheta)=:t\,\psi\left(\widetilde{\mu}_{-\vartheta}(x),x\right)
-\vartheta\cdot \varpi.
\end{equation}

Let us estimate each summand asymptotically by the stationary phase Lemma. On $\mathrm{supp}(\varrho_e)$,
$\partial_t\Psi(x,t,\vartheta)=\psi\left(\widetilde{\mu}_{-\vartheta}(x),x\right)=0$ if and only if
$\vartheta=\mathbf{b}_e$. Let us set $\vartheta=\eta+\mathbf{b}_e$, $\eta\sim \mathbf{0}$;
then $\widetilde{\mu}_{-\vartheta}(x)=\widetilde{\mu}_{-\eta}(x)$, and
in view of (\ref{eqn:T-g-local-action}) we have
$\partial_\vartheta\Psi(x,t,\mathbf{b}_e)=t\,\Phi(m)-\varpi$.

Let us set $\lambda_\varpi(m)=:\|\varpi\|/\|\Phi(m)\|$. Then by the above the following alternative holds.
If $m\not\in M_\varpi=\Phi^{-1}(\mathbb{R}_+\cdot m)$, then $\Psi(x,t,\vartheta)$ has no critical point in
$(t,\vartheta)$ on the support of $\varrho_e$ for any $e$,
so $\widetilde{\Pi}_{k\varpi}(x,x)=O\left(k^{-\infty}\right)$. If on the other hand $m\in M_\varpi$ then
$(t_e,\vartheta_e)=:\big(\lambda_\varpi(m),\mathbf{b}_e\big)$ is the only stationary point of
$\Psi(m,\cdot,\cdot)$ in $\mathbb{R}_+\times \mathrm{supp}(\varrho_e)$.

Assuming $m\in M_\varpi$ and that $\Phi$ is transversal to $\mathbb{R}_+\cdot \varpi$,
let us compute the Hessian matrix of $\Psi$ at each critical point.
Clearly, $\partial^2_{tt}\Psi=0$ identically. With $\vartheta=\eta+\mathbf{b}_e$, $\eta\sim \mathbf{0}$
by Lemma \ref{lem:more-refined-local-estimate} we have
\begin{eqnarray}
\label{eqn:phase-approximate-general-1}
\lefteqn{\Psi(x,t,\vartheta)=\Psi(x,t,\eta+\mathbf{b}_e)=t\,\psi\left(\widetilde{\mu}_{-\eta}(x),x\right)
-\eta\cdot\varpi-\mathbf{b}_e\cdot \varpi}\\
&=& \Big[t\,\psi\Big(x+\left(\eta\cdot\Phi(m)+O\left(\|\eta\|^3\right),
-\eta\cdot\xi_M(m)+O\left(\|\eta\|^2\right)\right),x\Big)
-\eta\cdot \varpi\Big] -\mathbf{b}_e\cdot \varpi. \nonumber
\end{eqnarray}
In view of (\ref{eqn:T-g-local-action}), the first line of
(\ref{eqn:phase-approximate-general-1}) implies
\begin{equation}
\label{eqn:hessian-t-vartheta}
\left.\partial^2_{t\vartheta}\Psi\right|_{\vartheta=\mathbf{b_e}}=
\left.\partial^2_{t\vartheta}\Psi\right|_{\eta=\mathbf{0}}=
\left.\partial_{\eta}\psi\left(\widetilde{\mu}_{-\eta}(x),x\right)\right|_{\eta=\mathbf{0}}=\Phi(m).
\end{equation}
In view of the discussion in \S 3 of \cite{sz}, the second line of (\ref{eqn:phase-approximate-general-1})
may be rewritten
\begin{eqnarray}
\label{eqn:phase-expand-general-2}
\lefteqn{\Psi(x,t,\eta+\mathbf{b}_e)}\\
&=&-\mathbf{b}_e\cdot \varpi+i t\,\Big\{\left[1-e^{i\eta\cdot\Phi(m)}\right]+\frac{1}{2}\,\|\eta\cdot \xi_M(m)\|^2
e^{i\eta\cdot\Phi(m)}+t\cdot O\left(\|\eta\|^3\right)
\Big\}-\eta\cdot \varpi
 . \nonumber
\end{eqnarray}
We obtain from (\ref{eqn:phase-expand-general-2}):
\begin{equation}
\label{eqn:phase-expand-general-3}
\left.\partial^2_{\vartheta_j\vartheta_k}\Psi\right|_{\vartheta=\mathbf{b_e},t=\lambda_\varpi(m)}=
i\,\lambda_\varpi(m)\,\Big[\Phi_j(m)\,\Phi_k(m)+\langle \xi_j,\xi_k\rangle_m\Big].
\end{equation}
Thus the Hessian of $\Psi$ at the critical point is
$$
H(\Psi)=:\left(
           \begin{array}{cc}
             0 & \Phi(m)^t \\
             \Phi(m) & i\,\lambda_\varpi(m)\,\Big[\Phi_j(m)\,\Phi_k(m)+\langle \xi_j,\xi_k\rangle_m\Big] \\
           \end{array}
         \right).
$$

This is the Hessian matrix at each critical point for every $m\in M_\varpi$; it is a smooth matrix valued
function on $M_\varpi$. To compute its determinant, let us remark that
$C(m)=:\big[\langle \xi_j,\xi_k\rangle_m\big]$ is positive semi-definite, and the scalar product it
defines on $\mathfrak{t}$
is positive definite on $\ker \big(\Phi(m)\big)$, by Corollary \ref{cor:transverse-conic-mm}.
Let $D$ be the smooth (symmetric, positive definite) matrix valued function defined in the discussion preceding
the statement of Theorem \ref{thm:higher-dim-case}.
Performing row operations and
applying Lemma \ref{lem:key-det-3} we get
\begin{eqnarray}
\label{eqn:phase-expand-general-4}
\lefteqn{\det \big(H(\Psi)\big)=\left|
           \begin{array}{cc}
             0 & \Phi(m)^t \\
             \Phi(m) & i\,\lambda_\varpi(m)\,C(m) \\
           \end{array}
         \right|}\\
         &=&\big(i\lambda_\varpi(m)\big)^{\mathrm{g}-1}\left|
           \begin{array}{cc}
             0 & \Phi(m)^t \\
             \Phi(m) & C(m) \\
           \end{array}
         \right|=i^{\mathrm{g}+1}\,\lambda_\varpi(m)^{\mathrm{g}-1}\,\|\Phi(m)\|^2\cdot \det\big(D(m)\big).
         \nonumber
\end{eqnarray}
Hence,
\begin{eqnarray}
\label{eqn:critical-hessian-det-0}
\det \left(\frac{k}{2\pi i}\,H(\Psi)\right)&=&
\left(\frac{k}{2\pi}\right)^{\mathrm{g}+1}\,
\lambda_\varpi(m)^{\mathrm{g}-1}\cdot\|\Phi(m)\|^2\,\det \left(D(m)\right)\nonumber\\
&=&\left(\frac{k}{2\pi}\right)^{\mathrm{g}+1}\,
\lambda_\varpi(m)^{\mathrm{g}-1}\cdot \|\Phi(m)\|^2\,\mathcal{D}(m)^2;
\end{eqnarray}
Applying the stationary phase Lemma, we end up with an asymptotic expansion in descending powers of
$k$ for the $e$-th summand in (\ref{eqn:variuospieces-general}), with leading term
\begin{eqnarray}
\label{eqn:asympt-expansion-varpi}
\lefteqn{e^{-ik\mathbf{b}_e\cdot\varpi}
\left(\sqrt{2}\pi\right)^{-(\mathrm{g}-1)}}\\
&&\cdot\left(\|\varpi\|\cdot\frac{k}{\pi}\right)^{\mathrm{d}+(1-\mathrm{g})/2}\,
\frac{1}{\mathcal{D}(m)}\,\left(\frac{1}{\|\Phi(m)\|}\right)^{\mathrm{d}+1+(1-\mathrm{g})/2}.
\nonumber
\end{eqnarray}
\end{proof}

As mentioned in the introduction, if $M_\varpi'\neq \emptyset$
there we can give a different description of the leading term.
In the present situation Lemma \ref{lem:key-det-4} may be restated as follows:
\begin{lem}
\label{lem:G-m-det}
For any $m\in M'_\varpi$, we have
$$
\|\Phi(m)\|^2_m\cdot \det\big(C(m)\big)=\|\Phi(m)\|^2\cdot \det\big(D(m)\big).
$$
In particular,
$G(m)=:\|\Phi(m)\|^2_m\cdot \det\big(C(m)\big)$ extends to a smooth function on $M_\varpi$,
uniformly bounded away from $0$.
\end{lem}

Hence if $m\in M'_\varpi$ we can rewrite (\ref{eqn:critical-hessian-det-0}) as
\begin{eqnarray}
\label{eqn:critical-hessian-det}
\det \left(\frac{k}{2\pi i}\,H(\Psi)\right)=
\left(\frac{k}{2\pi}\right)^{\mathrm{g}+1}\,
\lambda_\varpi(m)^{\mathrm{g}-1}\cdot\|\Phi(m)\|_m^2\,\det\big(C(m)\big).
\end{eqnarray}

\begin{lem}
\label{lem:eff-vol}
Given that $m\in M'$,
$\det\big(C(m)\big)^{1/2}=(2\pi)^{-\mathrm{g}}\,|T_m|\cdot V_{\mathrm{eff}}(m)$.
\end{lem}

\begin{proof}
Let $\Upsilon_m:\mathbb{T}\rightarrow M$ be the $\mathcal{C}^\infty$ map $m\mapsto \mu_g(m)$.
Then $C(m)$ represents
the Riemannian metric on $\mathbb{T}$ given by the pull-back under $\Upsilon_m$
of the Riemannian metric on the orbit, in the trivialization of the tangent bundle
$T\mathbb{T}$
given by $\big(\partial/\partial\vartheta_j\big)$. The volume density of the
pulled-back metric in the standard angular coordinates is therefore
$\det\big(C(m)\big)^{1/2}\,d\vartheta$.  Noting that $C(m)$ is constant
along the orbit,  the volume of $\mathbb{T}$ for the
associated Riemannian
density is
$$
\int_{-\pi}^\pi\cdots\int_{-\pi}^\pi\det\big(C(m)\big)^{1/2}\,d\vartheta=(2\pi)^{\mathrm{g}}\,
\det\big(C(m)\big)^{1/2}.
$$
On the other hand, $\Upsilon_m$ is a $|T_m|$-to-$1$ covering of the orbit through $m$, and
therefore the latter volume equals $|T_m|\cdot V_{\mathrm{eff}}(m)$.
\end{proof}

Therefore, for $m\in M'_\varpi$
the leading term (\ref{eqn:asympt-expansion-varpi}) may rewritten:
\begin{eqnarray*}
e^{-ik\mathbf{b}_e\cdot\varpi}
\left(\|\varpi\|\cdot\frac{k}{\pi}\right)^{\mathrm{d}+(1-\mathrm{g})/2}\,
\frac{2^{(\mathrm{g}+1)/2}\pi}{|T_m|\,V_{\mathrm{eff}}(m)\|\Phi(m)\|_m}\,
\left(\frac{1}{\|\Phi(m)\|}\right)^\mathrm{\mathrm{d}+(1-\mathrm{g})/2}.
\end{eqnarray*}

\section{Proof of Theorem \ref{thm:scaling-limit-general}}

\begin{proof}
To begin with, suppose by contradiction that
$x_k,y_k\in X$ are such that
$$
\max\big\{\mathrm{dist}_X\big(\mathbb{T}\cdot x_k,\mathbb{T}\cdot y_k),
\mathrm{dist}_X\big(\mathbb{T}\cdot x_k,X_\varpi)\big\}\ge C_1\,k^{\epsilon_1-1/2},
$$
but
$\widetilde{\Pi}_{k\varpi}(x_k,y_k)\ge C_1\,k^{-N}$ for some $C_1,\,N,\,\epsilon_1>0$.
After passing to a subsequence, we may assume that either
$\mathrm{dist}_X\big(\mathbb{T}\cdot x_k,\mathbb{T}\cdot y_k)\ge C_1\,k^{\epsilon_1-1/2}$,
or $\mathrm{dist}_X\big(\mathbb{T}\cdot x_k,X_\varpi)\ge C_1\,k^{\epsilon_1-1/2}$.

In the first case, there are two possibilities.
If $\mathrm{dist}_X\big(\mathbb{T}\cdot x_k,\mathbb{T}\cdot y_k)\ge \delta$ for some fixed
$\delta>0$, an argument along the lines of the proof of Lemma \ref{lem:1st-reduction}
shows that $\widetilde{\Pi}_{k\varpi}(x_k,y_k)=O\left(k^{-\infty}\right)$.

Suppose then $\mathrm{dist}_X\big(\mathbb{T}\cdot x_k,\mathbb{T}\cdot y_k)\rightarrow 0$;
we may then arrange that $\mathrm{dist}_X(x_k,y_k)\rightarrow 0$,
and perhaps after passing to a subsequence we may assume $x_k,y_k\rightarrow x$ for some
$x\in X$.
We may then estimate $\widetilde{\Pi}_{k\varpi}(x_k,y_k)$ by means of
(\ref{eqn:variuospieces-general}) and (\ref{eqn:defn-phase-general}),
with $(x_k,y_k)$ in place of $(x,x)$.
The argument at the end of the proof of Theorem \ref{thm:circle-case},
based on integration by parts in $dt$,
again shows that $\widetilde{\Pi}_{k\varpi}(x_k,y_k)=O\left(k^{-\infty}\right)$.
Thus we reach a contradiction.

Suppose next that $x_k\in X$ is a sequence with
$\mathrm{dist}_X(x_k,X_\varpi)>C_1\,k^{\epsilon_1-1/2}$,
$\widetilde{\Pi}_{k\varpi}(x_k,x_k)\ge C_1\,k^{-N}$.
After passing to a subsequence, we may assume $x_k\rightarrow x_\infty$ for
$k\rightarrow +\infty$. If $x_\infty\not\in X_\varpi$, then
$\widetilde{\Pi}_{k\varpi}(x_k,x_k)=O\left(k^{-\infty}\right)$ by Theorem \ref{thm:higher-dim-case},
absurd. Thus $x_\infty\in X_\varpi$. Let us set $m_k=:\pi(x_k)$,
$m_\infty=:\pi(x_\infty)$. Then
$\mathrm{dist}_M(m_k,M_\varpi)>C_1\,k^{\epsilon_1-1/2}$, $m_\infty\in M_\varpi$
and $m_k\rightarrow m_\infty$.

We may estimate asymptotically $\widetilde{\Pi}_{k\varpi}(x_k,x_k)$
using (\ref{eqn:variuospieces-general}) in the neighborhood of $x_\infty$,
with $(x,x)$ replaced by $(x_k,x_k)$. Thus $\widetilde{\Pi}_{k\varpi}(x_k,x_k)$
is given by an oscillatory integral with the phase
\begin{equation}
\label{eqn:defn-phase-general-k}
\Psi_k(t,\vartheta)=:\Psi(x_k,t,\vartheta)=:t\,\psi\left(\widetilde{\mu}_{-\vartheta}(x_k),x_k\right)
-\vartheta\cdot \varpi.
\end{equation}

Fix $\epsilon'\in (1/2-\epsilon_1,1/2)$
and a bump function $f\in \mathcal{C}^\infty_0\left(\mathbb{R}^{2\mathrm{d}+1}\right)$, identically $=1$
on a neighborhood of the origin;
an adaptation of the proof of Lemma \ref{lem:S-k''} shows that only a rapidly
decaying contribution is lost as $k\rightarrow +\infty$ by inserting in the amplitude of (\ref{eqn:variuospieces-general})
a cut-off function
of the form $f\Big (k^{\epsilon'}\,\big(x_k-\widetilde{\mu}_{-\vartheta}(x_k)\big)\Big)$ (expressed in any
given coordinate chart centered at $x_\infty$).

We then have:
\begin{eqnarray}
\label{eqn:partial-vartheta-k}
\partial_\vartheta\Psi_k(t,\vartheta)&=&t\,\partial_\vartheta\psi\left(\widetilde{\mu}_{-\vartheta}(x_k),x_k\right)
-\varpi\\
&=&t\,\Phi(m_k)-\varpi+O\left(k^{-\epsilon'}\right)=t\,\Phi(m_k)-\varpi+o\left(k^{\epsilon_1-1/2}\right).
\nonumber
\end{eqnarray}

We may find an open neighborhood $V$ of $m_\infty$ in $M_\varpi$, and a smoothly varying
system of preferred local coordinates (see \S \ref{subsect:heis}) centered at points $m\in V$,
such that $m+\mathbf{t}\in M_\varpi$ if $\mathbf{t}\in T_mM_\varpi$.
At each $m\in M_\varpi$, the corresponding preferred local chart determines
isomorphisms of Euclidean vector spaces $N_m\cong \mathbb{R}^{\mathrm{g}-1}$
($N_{m}$ is the normal space to $M_\varpi$
at $m$), $T_mM_\varpi\cong \mathbb{R}^{2\mathrm{d}+1-\mathrm{g}}$,
and
$$
\mathbb{R}^{2\mathrm{d}}\cong T_mM\cong T_mM_\varpi\oplus N_m\cong \mathbb{R}^{2\mathrm{d}+1-\mathrm{g}}\oplus \mathbb{R}^{\mathrm{g}-1}.
$$

\begin{lem}
\label{lem:t-k-n-k}
If $\mathbf{t}\in \mathbb{R}^{2\mathrm{d}+1-\mathrm{g}}$ and $\mathbf{n}\in \mathbb{R}^{\mathrm{g}-1}$
are sufficiently small, then
$$
m_\infty+(\mathbf{t}+\mathbf{n})=(m_\infty+\mathbf{t})+\mathbf{n}+O\big(\|\mathbf{t}\|\cdot \|\mathbf{n}\|).
$$
\end{lem}

\begin{proof}
$m_\infty+(\mathbf{t}+\mathbf{n})=(m_\infty+\mathbf{t})+\mathbf{n}$
when either $\mathbf{t}=\mathbf{0}$ or $\mathbf{n}=\mathbf{0}$.
\end{proof}

In preferred local coordinates centered at $m_\infty$, $m_k=m_\infty+\mathbf{v}_k$ for a
unique $\mathbf{v}_k\in \mathbb{R}^{2\mathrm{d}}$.
We can uniquely write $\mathbf{v}_k=\mathbf{t}_k+\mathbf{n}_k$,
where $\mathbf{t}_k\in \mathbb{R}^{2\mathrm{d}+1-\mathrm{g}}$, $\mathbf{n}_k\in \mathbb{R}^{\mathrm{g}-1}$.
By Lemma \ref{lem:t-k-n-k}, since $\mathbf{t}_k\rightarrow \mathbf{0}$
we then have $m_k=m_k'+\mathbf{n}_k+o(\|\mathbf{n}_k\|)$, where $m_k'=:m_\infty+\mathbf{t}_k\in M_\varpi$.
Furthermore, since preferred local coordinates are isometric at the origin, for $k\gg 0$ we have
\begin{equation}
\label{eqn:bound-on-n-k}
\|\mathbf{n}_k\|\ge \frac 12\,\mathrm{dist}_M(m_k,m_k')\ge \frac 13\,\mathrm{dist}_M(m_k,M_\varpi)
\ge \frac 13\,C_1\,k^{\epsilon_1-1/2}.
\end{equation}
Given any $\beta\in \mathfrak{t}^*$, let us write $\beta=\beta^\parallel+\beta^\perp$, where
$ \beta^\parallel\in \mathrm{span}(\varpi)$, $\beta^\perp\in \mathrm{span}(\varpi)^\perp$.
Then
\begin{eqnarray}
\label{eqn:phi-partial-vartheta-k}
t\Phi(m_k)-\varpi&=&t\,\Phi\big(m_k'+\mathbf{n}_k\big)-\varpi+o(\|\mathbf{n}_k)\\
&=&\Big[t\,\left(\Phi(m_k)'+d_{m_k'}\Phi(\mathbf{n}_k)^\parallel\right)-\varpi\Big]
+t\,d_{m_k'}\Phi(\mathbf{n}_k)^\perp+o(\|\mathbf{n}_k\|).\nonumber
\end{eqnarray}

Now the first summand on the second line of (\ref{eqn:phi-partial-vartheta-k}) is in $\mathrm{span}(\varpi)$,
therefore
\begin{equation}
\label{eqn:norm-of-difference}
\big\|t\Phi(m_k)-\varpi\big\|\ge t\,\left\|d_{m_k'}\Phi(\mathbf{n}_k)^\perp\right\|+o(\|\mathbf{n}_k\|).
\end{equation}

With the unitary identification $\mathbb{R}^{\mathrm{g}-1}\cong N_{m_k'}$,
we have $\mathbf{n}_k=J_{m_k'}\big(\eta_{kM}(m_k')\big)$ for a unique $\eta_k\in \ker\big(\Phi(m_k')\big)$
(Lemma \ref{lem:N-ker-Phi-transverse}), and
$\|\eta_k\|>D\,k^{\epsilon_1-1/2}$ for some $D>0$.
Again by Lemma \ref{lem:N-ker-Phi-transverse}, for any $m\in M_\varpi$ and $\eta\in \ker\big(\Phi(m)\big)\setminus \{0\}$
we have
$d_m\Phi\circ J_{m}\big(\eta_{M}(m)\big)\not\in \mathrm{span}\big(\Phi(m)\big)$;
hence, there exists $C_2>0$ such that
\begin{equation}
\label{eqn:bound-phi-perp}
\Big\|d_m\Phi\circ J_{m}\big(\eta_{M}(m)\big)^\perp\Big\|\ge C_2\,\|\eta\|
\,\,\,\,\,\,\,\,\,\,(m\in M_\varpi,\,\,\eta\in \ker\big(\Phi(m)\big)),
\end{equation}
where perpendicularity is to $\mathrm{span}\big(\Phi(m)\big)$ in $\mathfrak{t}^*$.
Thus, (\ref{eqn:norm-of-difference}) and (\ref{eqn:bound-phi-perp}) imply
\begin{equation}
\label{eqn:norm-of-difference-expl}
\big\|t\Phi(m_k)-\varpi\big\|\ge D'\,k^{\epsilon_1-1/2},
\end{equation}
since we may assume $t$ bounded away from $0$.

One can now argue by integration by parts in $d\vartheta$ that
$\widetilde{\Pi}_{k\varpi}(x_k,x_k)=O\left(k^{-\infty}\right)$, a contradiction.

Thus if $x_k\in X$ is a sequence with
$\mathrm{dist}_X(x_k,X_\varpi)>C_1\,k^{\epsilon_1-1/2}$, then
$\widetilde{\Pi}_{k\varpi}(x_k,x_k)=O\left(k^{-\infty}\right)$. If
$x_k,y_k\in X$ and $\mathrm{dist}_X(x_k,X_\varpi)>C_1\,k^{\epsilon_1-1/2}$ then
$$
\left|\widetilde{\Pi}_{k\varpi}(x_k,y_k)\right|\le \sqrt{\widetilde{\Pi}_{k\varpi}(x_k,x_k)}\cdot
\sqrt{\widetilde{\Pi}_{k\varpi}(y_k,y_k)}=O\left(k^{-\infty}\right),
$$
since the first factor on the right hand side is $O\left(k^{-\infty}\right)$.
\end{proof}

\section{Proof of Theorem \ref{thm:scaling-limit-general-precise}.}

\begin{proof}
Let us set $x_{jk}=x+\upsilon_j/\sqrt{k}$. Then
\begin{equation}
\label{eqn:equiv-proj-general-case-rescaled}
\widetilde{\Pi}_{k\varpi}(x_{1k},x_{2k})=
\frac{1}{(2\pi)^{\mathrm{g}}}\,\int_{-\pi}^\pi\cdots\int_{-\pi}^\pi
e^{-ik\vartheta\cdot \varpi}\,\widetilde{\Pi}\left(\widetilde{\mu}_{-\vartheta}(x_{1k}),x_{2k}\right)\,d\vartheta.
\end{equation}
Following the argument leading to (\ref{eqn:variuospieces-general}), and adapting the proof
of Lemma \ref{lem:S-k''}, we get
for some $\epsilon>0$
\begin{eqnarray}
\label{eqn:variuospieces-general-k-g}
\lefteqn{\widetilde{\Pi}_{k\varpi}(x_{1k},x_{2k})\sim\frac{k^{1-\mathrm{g}/2}}{(2\pi)^{\mathrm{g}}}
\cdot \sum_{e=1}^{r_m}e^{-ik\,\mathbf{b}_e\cdot \varpi}}\\
&&\cdot\,
\int_{1/2C}^{2C}\int_{\mathbb{R}^{\mathrm{g}}} e^{ik\Psi_k^{(e)}(x,t,\eta)}\,
s\left(k t,\widetilde{\mu}_{-\eta/\sqrt{k}-\mathbf{b}_e}(x_{1k}),x_{2k}\right)\,f\left(k^{-1/9}\,\|\eta\|\right)\,dt\,d\eta,
\nonumber
\end{eqnarray}
where now
\begin{equation}
\label{eqn:defn-phase-general-g}
\Psi^{(e)}_k(x,t,\eta)=:t\,\psi\left(\widetilde{\mu}_{-\eta/\sqrt{k}-\mathbf{b}_e}(x_{1k}),x_{2k}\right)
-\frac{1}{\sqrt{k}}\,\eta\cdot \varpi.
\end{equation}
Integration is $d\eta$ over a ball of radius $k^{1/9}$ in $\mathbb{R}^\mathrm{g}$.
In view of Corollary \ref{cor:yet-more-refined-local-estimate}, as in the derivation
of (\ref{eqn:action-expanded-1}) we obtain
\begin{eqnarray}
\label{eqn:action-expanded-g}
\lefteqn{\widetilde{\mu}_{-\eta/\sqrt{k}-\mathbf{b}_e}(x_{1k})}\\
&=&x+\left( \frac{1}{\sqrt{k}}\,\big(\eta\cdot\Phi(m)+\theta_1\big)+\frac{1}{k}\,\omega_m\big(\eta\cdot \xi_M(m),\mathbf{v}_1^{(e)}\big)
+B_3\left(\frac{\eta}{\sqrt{k}},\frac{\mathbf{v}}{\sqrt{k}}\right),\right.
\nonumber\\
&&\left.+\frac{1}{\sqrt{k}}\,\left(\mathbf{v}_1^{(e)}-\eta\, \xi_M(m)\right)+B_2\left(\frac{\eta}{\sqrt{k}},
\frac{\mathbf{v}}{\sqrt{k}}\right)\right),\nonumber
\end{eqnarray}
where now $B_j$ denotes a smooth function vanishing to
$j$-th order at the origin in $\mathbb{R}^{\mathrm{g}}\times \mathbb{C}^\mathrm{d}$.

Let us write $\mathfrak{t}=\mathrm{span}(\xi_1)\oplus \ker\big(\Phi(m)\big)$, where
$\xi_1\in \ker\big(\Phi(m)\big)^\perp$ has unit norm, and let
$(\xi_2,\cdots,\xi_\mathrm{g})$ be an orthonormal basis of $\ker\big(\Phi(m)\big)$;
assume without loss that $\langle\Phi(m)\xi_1\rangle =\|\Phi(m)\|$.
If $\eta=\sum_j\eta_j\xi_j=\eta_1\xi_1+\eta'$, then $\eta\cdot \Phi=\langle\Phi(m),\xi_1\rangle=\eta_1\,\|\Phi(m)\|$
and $\eta\cdot \varpi=\eta_1\,\varpi_1=\eta_1\,\|\varpi\|$.

Inserting this in (\ref{eqn:action-expanded-g}), we obtain from (\ref{eqn:variuospieces-general-k-g})
$\widetilde{\Pi}_{k\varpi}(x_{1k},x_{2k})\sim \sum_{e=1}^{r_m}S_e(k)$, where now the analogue of
(\ref{eqn:S-e-rescaled-1}) is
\begin{eqnarray}
\label{eqn:S-e-rescaled-g}
\lefteqn{S_e(k)\sim \frac{1}{(2\pi)^\mathrm{g}}\,e^{-ik\mathbf{b}_e}\,k^{1-\mathrm{g}/2}}\\
&&\cdot \int_{\mathbb{R}^{\mathrm{g}-1}}\left[\int_{1/2C}^{2C}\int_{-C\,k^{1/9}}^{Ck^{1/9}}e^{i\sqrt{k}\,\Upsilon(t,\eta_1,\theta)}\,
e^{G(m,\theta,\mathbf{v},\upsilon)}\,S_k(m,\theta,\mathbf{v},\upsilon)\,\mathrm{d}t\,\mathrm{d}\eta_1\right]\,d\eta',
\nonumber
\end{eqnarray}
with
\begin{equation}
\label{eqn:upsilon-g}
\Upsilon(t,\eta_1,\theta)=:t\,\big(\eta_1\,\|\Phi(m)\|+\theta_1-\theta_2\big)-\eta_1\,\|\varpi\|,
\end{equation}
and
\begin{eqnarray*}
\lefteqn{G(m,\eta,\mathbf{v},\upsilon)=:-t\,\big(\eta_1\,\|\Phi(m)\|+\theta_1-\theta_2\big)^2}\\
&&+it\,\omega_m\left(\eta\cdot\xi_M(m),\mathbf{v}_1^{(e)}\right)+t\,\psi_2\left(\mathbf{v}_1^{(e)}-\eta\cdot \xi_M(m),\mathbf{v}_2\right)\,e^{i\,A_{ek}(\eta,\upsilon_1,\upsilon_2)};
\end{eqnarray*}
we have $\Re (G)\le -c\,\|\eta\|^2$ for some $c>0$.

The inner integral in (\ref{eqn:S-e-rescaled-g}) may be estimated as in the last part of the proof
of Theorem \ref{thm:circle-case}. We end up with an asymptotic expansion for $S_e(k)$ in descending powers of
$k^{1/2}$, and an $N$-th step remainder bounded by $C_N\,e^{-c\,\|\eta'\|^2}\,k^{-aN}$ for some
$a>0$ (and the integration takes place over a ball of radius $O\left(k^{1/9}\right)$); the constant $a$ is uniformly
bounded away from $0$ when we restrict $\upsilon$ to some vector subspace of $T_xX$ having zero intersection to
$\mathfrak{t}_X(x)$.
The leading order term is
\begin{eqnarray}
\label{eqn:leading-term-with-integral}
\lefteqn{\frac{1}{(2\pi)^{\mathrm{g}-1}}\,k^{\mathrm{d}+(1-\mathrm{g})/2}e^{-ik\,\mathbf{b}_e\cdot \varpi-i\sqrt{k}\,(\theta_1-\theta_2)/\|\Phi(m)\|}\,
}\\
&&\cdot \left(\frac{\lambda_\varpi(m)}{\pi}\right)^\mathrm{d}\,\frac{1}{\|\Phi(m)\|}\,
\cdot\int_{\mathbb{R}^{\mathrm{g}-1}} e^{G_0(\eta',\upsilon_1,\upsilon_2)}\,d\eta',   \nonumber
\end{eqnarray}
where, setting $\eta'\cdot \xi'=:\sum_{j=2}^\mathrm{\mathrm{g}}\eta_j\,\xi_j$,
\begin{eqnarray*}
G_0(\upsilon_1,\upsilon_2)&=:&\lambda_\varpi(m)\,\left[i\,\omega_m\left(\frac{\theta_2-\theta_1}{\|\Phi(m)\|}\,\xi_{1M}(m)+\eta'\cdot\xi'_M(m),
\mathbf{v}_1^{(e)}\right)\right.\\
&&\left.+
\psi_2\left(
\mathbf{v}_1^{(e)}-
\left(\frac{\theta_2-\theta_1}{\|\Phi(m)\|}\,\xi_{1M}(m)+\eta'\cdot\xi'_M(m)\right),
\mathbf{v}_2\right)\right].
\end{eqnarray*}
If $\upsilon_j=(0,\mathbf{v}_j)$ with $\mathbf{v}_j\in N_m$ for $j=1,2$, (\ref{eqn:leading-term-with-integral}) specializes to
\begin{eqnarray}
\label{eqn:leading-term-with-integral-0}
\frac{k^{\mathrm{d}+(1-\mathrm{g})/2}}{(2\pi)^{\mathrm{g}-1}}\,e^{-ik\,\mathbf{b}_e\cdot \varpi}\,
\cdot \left(\frac{\lambda_\varpi(m)}{\pi}\right)^\mathrm{d}\,\frac{1}{\|\Phi(m)\|}\,
\cdot\int_{\mathbb{R}^{\mathrm{g}-1}} e^{G_0(\eta',\upsilon)}\,d\eta',
\end{eqnarray}
where
\begin{eqnarray*}
G_0(\upsilon)&=:&\lambda_\varpi(m)\,\left[i\,\omega_m\left(\eta'\cdot\xi'_M(m),
\mathbf{v}_1^{(e)}\right)+
\psi_2\left(
\mathbf{v}_1^{(e)}-\eta'\cdot\xi'_M(m),
\mathbf{v}_2\right)\right]\\
&=&\lambda_\varpi(m)\,\left[-i\,\omega_m\left(\mathbf{v}_1^{(e)},\mathbf{v}_2\right)-\frac 12\,\|\mathbf{v}_1^{(e)}-\mathbf{v}_2\|^2\right]\\
&&+\lambda_\varpi(m)\,\left[i\,
\omega_m\left(\eta'\cdot\xi'_M(m),\mathbf{v}_1^{(e)}+\mathbf{v}_2\right)-\frac 12
\left\|\eta'\cdot\xi'_M(m)\right\|^2\right].
\end{eqnarray*}
With the coordinate change $\nu=\sqrt{\lambda_\varpi(m)}\cdot\eta'$, and using that
$\omega_m(\mathbf{u},\mathbf{w})=-g_m\big(\mathbf{u},J_m(\mathbf{w})\big)$, we obtain
\begin{eqnarray*}
\lefteqn{\int_{\mathbb{R}^{\mathrm{g}-1}} e^{\lambda_\varpi(m)\,\left[i\,
\omega_m\left(\eta'\cdot\xi'_M(m),\mathbf{v}^{(e)}_1+\mathbf{v}_2\right)-\frac 12
\left\|\eta'\cdot\xi'_M(m)\right\|^2\right]}\,d\eta'}\\
&=&\lambda_\varpi(m)^{-(\mathrm{g}-1)/2}\int_{\mathbb{R}^{\mathrm{g}-1}} e^{
-i\,g_m\left(\nu\cdot\xi'_M(m),\sqrt{\lambda_\varpi(m)}\cdot(\mathbf{v}_1^{(e)}+\mathbf{v}_2)\right)-\frac 12
\left\|\nu\cdot\xi'_M(m)\right\|^2}\,d\nu.
\end{eqnarray*}
Now recall that $(\xi_2',\ldots,\xi'_{\mathrm{g}})$ is an orthonormal basis of $\ker\big(\Phi(m)\big)$,
while $\|\cdot\|$, the standard Euclidean norm on $\mathbb{R}^{\mathrm{g}-1}$, represents in Heisenberg local
coordinates the restriction to
$\mathrm{ev}_m\big(\ker\big(\Phi(m)\big)\big)\cong N_m\subseteq T_mM$ of the Riemannian norm
$\|\cdot \|_m$. Given this, we have
\begin{eqnarray*}
\lefteqn{\lambda_\varpi(m)^{-(\mathrm{g}-1)/2}\int_{\mathbb{R}^{\mathrm{g}-1}} e^{
-i\,g_m\left(\nu\cdot\xi'_M(m),\sqrt{\lambda_\varpi(m)}\cdot(\mathbf{v}_1^{(e)}+\mathbf{v}_2)\right)-\frac 12
\left\|\nu\cdot\xi'_M(m)\right\|^2}\,d\nu}\\
&=&\lambda_\varpi(m)^{-(\mathrm{g}-1)/2}\cdot \frac{1}{\mathcal{D}(m)}\cdot
(2\pi)^{(\mathrm{g}-1)/2}\cdot e^{-\lambda_\varpi(m)\,\|\mathbf{v}_1^{(e)}+\mathbf{v}_2\|^2/2}.
\end{eqnarray*}
Summing up, we obtain for $S_e(k)$ an asymptotic expansion in descending powers of $k^{1/2}$, with leading order term
\begin{eqnarray*}
\lefteqn{\frac{1}{(\sqrt{2}\pi)^{\mathrm{g}-1}}\cdot \left(\lambda_\varpi(m)\cdot\frac{k}{\pi}\right)^{\mathrm{d}+(1-\mathrm{g})/2}}\\
&&\cdot \frac{e^{-ik\mathbf{b}_e\cdot \varpi}}{\|\Phi(m)\|\,\mathcal{D}(m)}\cdot
e^{-i\,\lambda_\varpi(m)\,\omega_m\left(\mathbf{v}_1^{(e)},\mathbf{v}_2\right)-\lambda_\varpi(m)\,
\big(\|\mathbf{v}_1^{(e)}\|^2+\|\mathbf{v}_2\|^2\big)}.
\end{eqnarray*}
Notice that for $\mathbf{v}_1=\mathbf{v}_2=0$ we recover the leading term of Theorem
\ref{thm:higher-dim-case} (where the corresponding expansion whose proved to be in descending powers of $k$).
\end{proof}

\subsection{Proof of Corollary \ref{cor:asymp-bound-dim-gen-case}}

\begin{proof}
Under our assumptions,
$M_\varpi=\Phi^{-1}(\mathbb{R}_+\cdot \varpi)$ is a compact submanifold
of $M$, of (real) codimension $\mathrm{g}-1$. Let us set $X_\varpi=:\pi^{-1}(M_\varpi)$.

For any $x\in X_\varpi$, we can find a neighborhood $U$ of $x$ in $X_\varpi$ and
a smoothly varying family of Heisenberg local coordinates centered at $x'\in U$,
that we shall denote $\gamma_{x'}(\theta,\mathbf{v})=x'+(\theta,\mathbf{v})$
(\S \ref{subsect:heis}). We may assume without loss that
$$
r_\beta(x')+(\theta,\mathbf{v})=x'+(\beta+\theta,\mathbf{v})
$$
where defined, and that $\gamma_{x'}$
is adjusted to $X_\varpi$, in the following sense.

For any $x'\in X_\varpi$, we have
$T_{x'}X=T_{x'}X_\varpi\oplus N_{x'}$, where $N_{x'}$ is the normal space of $X_\varpi$
in $X$ at $x'$; $N_{x'}$ is naturally unitarily isomorphic to the normal space of
$M_\varpi$ in $M$ at $m'=\pi(x')$. We may then assume that, under the unitary isomorphisms
$T_{x'}X\cong \mathbb{R}\oplus \mathbb{R}^{2\mathrm{d}}\cong
\left(\mathbb{R}\oplus \mathbb{R}^{2\mathrm{d}-(\mathrm{g}-1)}\right)\oplus \mathbb{R}^{\mathrm{g}-1}$
induced by $\gamma_{x'}$,
$$
T_{x'}X'_\varpi\cong \mathbb{R}\oplus \mathbb{R}^{2\mathrm{d}-(\mathrm{g}-1)},\,\,\,
N_{x'}\cong \mathbb{R}^{\mathrm{g}-1}.
$$

If $x'\in U$ and $\mathbf{v}\in \mathbb{R}^{\mathrm{g}-1}$ is suitably small, set
$x'+\mathbf{v}=:\gamma_{x'}\big(0,(\mathbf{0},\mathbf{v})\big)$.
For some $\delta>0$,
the map
$$
(x',\mathbf{v})\in U\times B_{\mathrm{g}-1}(\mathbf{0};2\delta)\mapsto x'+\mathbf{v}\in X
$$
is a diffeomorphism onto its image ($B_{\mathrm{g}-1}(\mathbf{0};\varepsilon)$ is the open ball of center the origin
and radius $\varepsilon$ in $\mathbb{R}^{\mathrm{g}-1}$).
We may assume without loss that $U$ is (the image of) a coordinate chart $x=x(u)$
for $X_\varpi$, and
let $du$ be the standard measure in local coordinates on $U$. Then under the previous
diffeomorphism  $(u,\mathbf{v})$
are local coordinates on $X$ centered at $x$, and
$dV_X=\mathcal{V}\big(x(u),\mathbf{v}\big)\,du\,d\mathbf{v}$ for a smooth positive function
$\mathcal{V}$.
By construction, and definition of
Heisenberg local coordinates, the volume form on $X_\varpi$, expressed
in the local coordinates on $U$, is $dV_{X_\varpi}=\mathcal{V}\big(x(u),\mathbf{0}\big)\,du$.

Let $(V_a)$ be a finite open cover of $M_\varpi$, such that the previous construction can be carried out on
each $U_a=:\pi^{-1}(U_a)$. Also, let $(\gamma_a)$ be a partition of unity on $M_\varpi$ subordinate to
$(V_a)$, implicitly pulled back to $X_\varpi$.

Finally let $\varsigma$ be a smooth bump function on $\mathbb{R}^{\mathrm{g}-1}$, vanishing for
$\|\mathbf{v}\|\ge 2\delta$, and identically equal to $1$ for $\|\mathbf{v}\|\le \delta$.

We may assume that for any $a$ and $x\in U_a$,
$x+\mathbf{v}$ is at distance $\ge \|\mathbf{v}\|/2$ from $X_\varpi$
if $\mathbf{v}\in \mathbb{R}^{\mathrm{g}-1}$, $\|\mathbf{v}\|\le \delta$.
Since $\widetilde{\Pi}_{k\varpi}(x,x)=O\left(k^{-\infty}\right)$ as $k\rightarrow +\infty$ when $x\not\in
X_\varpi$, perhaps after neglecting a negligible contribution integration may be restricted to
a tubular neighborhood of $X_\varpi$. Therefore as $k\rightarrow +\infty$
\begin{eqnarray}
\label{eqn:integrazione-su-X_r}
\int_{X_r}\widetilde{\Pi}_{k\varpi}(x,x)\,dV_X(x)&\sim&
\sum_a\int_{U_a}\gamma_a\big(x(u)\big)\,P_{ka}\big(x(u)\big)\,du\nonumber\\
&=&\sum_a\int_{U_a}\gamma_a(x)\,\frac{P_{ka}(x)}{\mathcal{V}_a\big(x,\mathbf{0}\big)}\,dV_{X_\varpi}(x)
\end{eqnarray}
where, for $x\in U_a$,
\begin{equation}
\label{eqn:pka}
P_{ka}\big(x\big)=:\int_{B(\mathbf{0};2\delta)}\varsigma(\mathbf{v})\,\widetilde{\Pi}_{k\varpi}
\big(x+\mathbf{v},x+\mathbf{v}\big)\,\mathcal{V}_a(x,\mathbf{v})\,d\mathbf{v};
\end{equation}
the index $a$ appears in the right hand side of (\ref{eqn:pka}) through the moving system of
Heisenberg local coordinates on $U_a$, and
$\mathcal{V}_a(x,\mathbf{0})\,du=dV_{X_\varpi}\big(x(u)\big)$.

Let us now study the asymptotics of $P_{ka}\big(x\big)$.
In view of Theorem \ref{thm:scaling-limit-general}, only a rapidly decaying contribution is lost
if the integrand in (\ref{eqn:pka}) is multiplied by a cut-off of the form
$\gamma_k(\mathbf{v})=:\gamma \left(k^{1/2-\epsilon_1}\|\mathbf{v}\|\right)$, where $\gamma$ is some compactly supported bump function
identically equal to $1$ in a neighborhood of $0$. If we now adopt the rescaling
$\mathbf{v}\rightsquigarrow \mathbf{v}/\sqrt{k}$, we end up with
\begin{eqnarray}
\label{eqn:pka-estimate}
\lefteqn{P_{ka}\big(x\big)=k^{-(\mathrm{g}-1)/2}}\\
&&\cdot\int_{\mathbb{R}^{\mathrm{g}-1}}\gamma_k(\mathbf{v})\,\widetilde{\Pi}_{k\varpi}
\left(x+\frac{\mathbf{v}}{\sqrt{k}},x+\frac{\mathbf{v}}{\sqrt{k}}\right)\,\mathcal{V}_a\left(x,\frac{\mathbf{v}}{\sqrt{k}}\right)
\,d\mathbf{v}.\nonumber
\end{eqnarray}
Integration is now over a ball of radius $O\left(k^{\epsilon_1}\right)$.

Using Theorem \ref{thm:scaling-limit-general-precise}, and then integrating in $d\mathbf{v}$, we obtain
for $P_{ka}\big(x\big)$ an asymptotic expansion in descending powers of $k^{1/2}$,
whose leading term is generically given by
$$
\frac{1}{(2\pi)^{\mathrm{g}-1}}\,\left(\|\varpi\|\cdot \frac k\pi\right)^{\mathrm{d}+1-\mathrm{g}}
\cdot \frac{1}{\mathcal{D}(m)}\,\left(\frac{1}{\|\Phi(m)\|}\right)^{\mathrm{d}+2-\mathrm{g}}.
$$
One can then argue as in the proof of Corollary \ref{cor:dim-circle-case}.

\end{proof}

\end{document}